\UseRawInputEncoding
\documentclass[12pt]{article}

\usepackage{amsfonts, amssymb, amscd, amsthm, amsmath, graphics}

\input xy
\xyoption{all}

\theoremstyle{plain}
{
\swapnumbers
 \newtheorem{thm}{Theorem}
 [subsection]
 
 \newtheorem{cor}[thm]{Corollary}
 \newtheorem{lem}[thm]{Lemma}
 \newtheorem{prop}[thm]{Proposition}
 \newtheorem{rem}[thm]{Remark}

 \newtheorem{exa}[thm]{Example}
 \newtheorem{que}[thm]{Question}
 \newtheorem{defi}[thm]{Definition}
}

\theoremstyle{definition}
{
\swapnumbers
 \newtheorem{sssection}[thm]{}
}
\renewcommand{\subsubsection}{\sssection}

\newcommand{\zz}{{\Bbb Z}}

\newcommand{\aaa}{{\Bbb A}}

\newcommand{\hh}{{\Bbb H}}

\newcommand{\ddim}{\operatorname{dim}}
\newcommand{\ddet}{\operatorname{det}}
\newcommand{\ddeg}{\operatorname{deg}}
\newcommand{\kker}{\operatorname{Ker}}

\newcommand{\Hom}{\operatorname{Hom}}
\newcommand{\op}[1]{\operatorname{#1}}
\newcommand{\kbar}{\overline{k}}

\newcommand{\ffi}{\varphi}

\newcommand{\eps}{\varepsilon}

\newcommand{\la}{\langle}
\newcommand{\ra}{\rangle}
\newcommand{\lva}{\langle\!\langle}   
\newcommand{\rva}{\rangle\!\rangle}   
\newcommand{\row}{\rightarrow}
\newcommand{\llow}{\longleftarrow}
\newcommand{\low}{\leftarrow}
\newcommand{\lrow}{\longrightarrow}
\renewcommand{\leq}{\leqslant}
\renewcommand{\geq}{\geqslant}

\newcommand{\calm}{{\cal M}}
\newcommand{\hm}{\operatorname{H}_{\calm}}
\newcommand{\nichego}[1]{}

\newcommand{\ov}[1]{\overline{#1}}

\newcommand{\wt}[1]{\widetilde{#1}}

\newcommand{\smF}[1]{{\operatorname {Sm}_{#1}}}

\newcommand{\cv}{{\cal V}}
\newcommand{\cu}{{\cal U}}
\newcommand{\cy}{{\cal Y}}

\newcommand{\bzar}[1]{B{#1}}
\newcommand{\btzar}[1]{\wt{B}{#1}}
\newcommand{\bttzar}[1]{{\hat{B}}{#1}}
\newcommand{\bzart}[1]{\wt{B}{#1}}
\newcommand{\bet}[1]{B{#1}_{et/Nis}}

\newcommand{\betO}[1]{B{#1}_{et}}
\newcommand{\bgm}[1]{B{#1}_{gm}}
\newcommand{\ezar}[1]{E{#1}}

\newcommand{\egm}[1]{E{#1}_{gm}}
\newcommand{\hk}{{\cal H}(k)}

\newcommand{\hF}[1]{{\cal H}(#1)}
\newcommand{\hH}{{\cal H}}
\newcommand{\hsF}[1]{{\cal H}_s(#1)}
\newcommand{\hs}{{\cal H}_s}

\newcommand{\dmkDVA}{\op{DM}^{-}_{eff}(k;\zz/2)}
\newcommand{\dm}[1]{\op{DM}^{-}_{eff}(#1)}
\newcommand{\dmDVA}[1]{\op{DM}^{-}_{eff}(#1;\zz/2)}
\newcommand{\dmR}[2]{\op{DM}^{-}_{eff}(#1,#2)}
\newcommand{\dmcohR}[2]{\op{DM}^{-}_{coh}(#1,#2)}
\newcommand{\hiq}[1]{{\cal{X}}_{#1}}

\newcommand{\hiqt}[1]{\widetilde{{\cal{X}}}_{#1}}
\newcommand{\gm}{{\Bbb G}_m}
\newcommand{\ax}[1]{\alpha_{#1}}
\newcommand{\car}{{\cal R}}
\newcommand{\dcar}{\partial{\cal R}}
\newcommand{\tate}[1]{T_{#1}}
\newcommand{\laa}{{\mathbf{R}}}

\newcommand{\hsp}{\hspace{2mm}}
\newcommand{\LRw}{\hsp\Leftrightarrow\hsp}

\newcommand{\Qed}{\hfill$\square$\smallskip}

\nichego{
\newtheorem{prop}{Proposition}[section]{\bf}{\it}
\newtheorem{thm}[prop]{Theorem}{\bf}{\it}
\newtheorem{lem}[prop]{Lemma}{\bf}{\it}
{\bf}{\it}
\newtheorem{defi}[prop]{Definition}{\bf}{\it}
{\bf}{\it}
{\bf}{\it}
\newtheorem{exa}[prop]{Example}{\bf}{\it}
\newtheorem{rem}[prop]{Remark}{\bf}{}
\newtheorem{que}[prop]{Question}{\bf}{}

{\bf}{\it}
\newtheorem{cor}[prop]{Corollary}{\bf}{\it}
{\bf}{\it}
}

\newcommand{\Spec}{\operatorname{Spec}}
\newcommand{\Field}{k}
\newcommand{\charac}{\operatorname{char}}

\newcommand{\Shv}{\operatorname{Shv}}

\newcommand{\Spc}{\operatorname{Spc}}

\newcommand{\cosk}{\operatorname{cosk}}
\newcommand{\Supp}{\operatorname{Supp}}
\newcommand{\Img}{\operatorname{Im}}

\newcommand{\TTT}{{\operatorname{T}}}
\newcommand{\EEE}{\operatorname{E}}

\newcommand{\rk}{\operatorname{rk}}
\newcommand{\Aut}{\operatorname{Aut}}
\newcommand{\Iso}{\operatorname{Iso}}

\newcommand{\PST}{\operatorname{PST}}

\newcommand{\sbt}{\,\begin{picture}(-1,1)(-1,-3)\circle*{3}\end{picture}\,\ }
\newcommand{\sSpc}{\operatorname{sSpc}}

\begin{document}

\title {Subtle Characteristic Classes}

\author{A.\,Smirnov\footnote{Partially supported by RFFI (grant
13-01-00429 À).}, A.\,Vishik\footnote{Partially supported by EPSRC (RM grant
EP/G032556/1).}}

\date{}

\maketitle


\begin{abstract}
\noindent We construct new {\it subtle Stiefel--Whitney classes} of quadratic forms.
These classes are much more
informative than the ones introduced by Milnor. In particular, they see all the powers of the
fundamental ideal of the Witt ring, contain the Arason invariant and it's higher analogues.
Moreover, the new classes allow to treat
the $J$-invariant of quadrics. This invariant, introduced in \cite{CGQG},
has been so far completely isolated from
characteristic classes. In addition, our classes allow to describe explicitly the
structure of some motives associated with quadratic forms.
\end{abstract}


\section{Introduction}

This work grew out of attempts to develop a sufficiently generic motivic homotopic approach
to the classification of algebro-geometric structures and to apply it to quadratic forms.

Here we restrict our considerations to the classification of torsors of algebraic
groups over a field. Usually such torsors are classified by the \'etale cohomology,
whereas the Zariski topology is quite inadequate, especially over a field. The situation changes
if one consider the large Zariski site instead of the small one, since then torsors are split by
appropriate schemes. Here some new fenomenon
appears.
Namely,
the sheaf represented by
a torsor
does not surject to the base. Therefore, when
passing
to sheaves
we get
a torsor not over the
base, but over it's part.
It is called
the support of
the corresponding sheaf.
Consequently, one obtains the
natural arrow from this support to the respective classificator.
These arrows
are
very interesting
invariants of
torsors
which should play an important role
in the motivic homotopic approach to the classification.
Actually
we get the whole family of classificators, supports,
and arrows indexed by
appropriate topologies. We mainly consider
the Nisnevich topology
since
in this case we have
a well-developed motivic homotopy theory,
namely the theory of F.\,Morel and V.\,Voevodsky.
Our results demonstrate that
this way we get
sufficiently rich
invariants.

As soon as we have the above
arrow
to the classificator
we can apply any cohomology functor to it getting
a
homomorphism
from the cohomology
of the classificator to the
cohomology
of the support.
This invariant, introduced in $\S\ref{ClassificationProblemS}$ for arbitrary groups, is the main
object we study in the current paper.
The main novelty of our approach is that we consider the Nisnevich classificator where traditionally an
\'etale one was studied, and that we (respectively) substitute the "point" by the "support of the torsor".
This permits to produce invariants which are much more informative.


In $\S\ref{CohomologyOfBOS}$ we
focus on the orthogonal case.
The torsors here are in one-to-one correspondence with quadratic forms. As
the cohomology functor
we
take the
motivic
$\zz/2$-cohomology
and
compute
it
for
the Nisnevich classificator $BO(n)$.
Theorem \ref{u} says
that the result is the polynomial
ring over motivic cohomology of the base field with canonical generators $u_1,\ldots,u_n$.
These are our
{\it subtle }Stiefel--Whitney classes.
The description we get is much
simpler
than that for the motivic cohomology of the \'etale classificator
(this cohomology is computed by
N.\,Yagita
in \cite{YaCF}).
The classes $u_i$'s by means of
pullback,
associated with
the canonical arrow to the classificator $BO(n)$,
give
the
subtle Stiefel--Whitney
classes of
any
individual quadratic form $q$.
They
take values in the
motivic $\zz/2$-cohomology
of the Chech simplicial scheme
related
to the
torsor $X_q$.

Originally, in the context of quadratic forms,
the Stiefel--Whitney classes $w_i$ were introduced algebraically by
J.\,Milnor
in \cite{Mil}
(the paper where the Milnor's K-theory was also introduced).
They
allow
to identify $K^M_2(\Field)/2$ with
the second component of the graded Witt ring.
The
classes
$w_i$
were interpreted as
pullbacks
of
some
\'etale cohomology classes of the
corresponding
classificator by
H.\,Esnault, B.\,Kahn, and E.\,Viehweg
in \cite{EKV} (see also J.\,F.\,Jardine \cite{Jar}).
The drawback of the classes of
J.\,Milnor
though is that they are trivial on
$I^3(\Field)$
(provided $(-1)$ is a square in $\Field$), which makes their use for the classification
of quadratic forms quite limited.
In contrast, our classes
are non-trivial on any power $I^n$ of the fundamental ideal.
Moreover,
they distinguish if the form is in $I^n$, or not
(see Theorem \ref{In}), and so distinguish
the triviality of
torsors
(see Cor.\,\ref{Utriv}).
We establish explicit
relations
between our {\it subtle} classes and classical ones as well as with
the Chern classes. It appears that our classes are "approximately" equal to square roots of
the
Chern classes,
while are obtained from $w_i$'s by "dividing" the latter by some powers of
$\tau\in\hm^{0,1}(\Spec \Field,\zz/2)$.
In particular, the topological realization
identifies $u_i$ and $w_i$ with the topological
Stiefel--Whitney
class.
Since
multiplication by $\tau$ is usually a very non-injective map in
the motivic cohomology,
we see
why Milnor's classes loose so much information.

We compute the action of the Steenrod operations on our classes which appears to be as simple as
in the topological case
provided $(-1)$ is a square in $\Field$ (see Prop.\,\ref{Stinrodu}).
We
also describe
the behavior of the
subtle classes under
addition of
quadratic forms
(see Prop.\,\ref{Olmn}).
After that it becomes possible to compute these classes effectively.
In particular, we describe them completely in the case of a Pfister form
(see Theorem \ref{Pfister}).
Here we are exploiting the fact that
the
motivic cohomology of the respective Chech simplicial scheme is
known in this case. This computation
enables us to show that
the subtle Stiefel--Whitney classes do
remember
the
Arason invariant and higher invariants identifying $I^r/I^{r+1}$ with $K^M_r(\Field)/2$.
This is done in Theorem \ref{en}.

Finally, we use our classes to describe the motive of the torsor $X_q$ and the motive
of the highest quadratic
Grassmannian as an explicit extension of
twisted
motives of the simplicial
Chech
scheme, related to $X_q$
(see
Theorem\,\ref{XhiqX} and Theorem \ref{GdhiqEV}). It appears that these motives have poly-binary
structure (are
tensor
products of motives each of which over algebraic closure decomposes as a sum of
just two Tate-motives). This
allow
to relate
subtle Stiefel--Whitney
classes with the
$J$-invariant of $q$
(see \ref{u->J} and \ref{Jcond}).
Thus, our classes connect
Stiefel--Whitney
classes with the $J$-invariant.
These areas
were previously
completely isolated from each other.

Of course, there are
similar
subtle
versions of
other
characteristic classes
and these
would
be a cornestone of
the classification of
respective
structures.
As for quadratic forms,
the
subtle Stiefel--Whitney
classes should serve as a
zero-order
step in the
homotopic classification of quadratic forms in the same way as the $J$-invariant is the
zero-order,
but the most important
step of the $EDI$-invariant (see \cite{u}). Thus, the task now
is to build
those next layers on top of what we have.
Probably, it will involve some sort of higher subtle characteristic classes.

We
thank M.\,Schlichting
for very useful discussions.
We are very grateful to the Institute for Advanced Study at Princeton, and to the
EPSRC Responsive Mode grant EP/G032556/1 for the support which made it possible for our
collaboration to happen.


\section{Torsors and Their Classificators}
\label{ClassificationProblemS}

Let $\Field$ be a field, $S=\Spec\Field$,
and $G$
be a
smooth linear algebraic group over $S$.
We are
going
to consecutively describe the problem of classification of $G$-torsors
on the pre-homotopic level, that is, on the level of spaces, then on the level of
simplicial homotopic category
$\hs=\hs(S)$, and finally on
the level of homotopic category
$\hH=\hH(S)$.

\subsection{Algebraic torsors}
\label{ClassificationProblemStatementSS}

The use of the notion of $G$-torsor in various situations will require a
certain degree of precision from us. Let us achieve it through the following
chain of definitions.

\begin{defi}
\label{AlgTorsor}
An "algebraic $G$-torsor" is a non-empty $S$-scheme $P$ together with
an
action
$\alpha: G\times_S P \to P$
such
that the map $(\alpha, \pi_P):G\times P \to P\times P$
is an isomorphism. Here $\pi_P$ is the projection $G\times P \to P$.
\end{defi}

\subsubsection{Example.}
\label{Example100}
In our main example
$char(\Field)\neq 2$, $q$
is
a
non-degenerate quadratic form over
$\Field$, and $G$
is
the orthogonal group of $q$. With each non-degenerate quadratic form
$p$
with
$\rk p =\rk q$
one can associate the algebraic $G$-torsor $X$ represented by
the functor
$\Iso(p\row q)$. It is known that each algebraic $G$-torsor can be obtained
this
way
and that $p$ can be recovered from $X$ and $G$-action.

\subsection{Torsor
classificators
on
pre-homotopic level}
\label{TorsorsClassificationOnSpaceLevelSS}

Let
$\TTT$ be a site represented by the category $\smF{S}$ with
the
Nisnevich topology
and
$\Shv/\TTT$ be the category of sheaves of
sets
on $\TTT$. Our category of spaces $\Spc$
is defined as the category of simplicial objects in $\Shv/\TTT$.

We
identify the group $G$ with the respective representable sheaf of groups,
that is,
with
a
group in $\Shv/\TTT$, and also with the constant group in $\Spc$.
In addition, if $X$ is a $G$-scheme, then the corresponding sheaf in $\Shv/\TTT$ and the
space $\Spc$ inherit the action of $G$. At the same time, it is not true that, this way,
an algebraic torsor
always
produces
a torsor. To clarify this
we
recall the definition of
torsor in
a
Grothendieck
topos $\EEE$.

\begin{defi}
\label{FormalTorsor-Torsor}
An object $P\in \EEE$ together with
an
action $\alpha: G\times P \to P$ is called a "formal $G$-torsor"
if $(\alpha, \pi_P):G\times P \to P\times P$ is an isomorphism, where $\pi_P$ is the projection.
A formal torsor $(P, \alpha)$ is called a "torsor" if the unique map $P \to 1$ to the final object
is
an epimorphism.
\end{defi}

Clearly, each algebraic $G$-torsor $X$ induces formal $G$-torsors (also denoted as $X$) in
$\Shv/\TTT$ and $\Spc$. Setting aside the question of which formal torsors can be obtained this
way, we mention only that the replacement of an algebraic torsor by a formal one
looses
no information in the sense that it does not glue objects (is a conservative functor).

\subsubsection{Example.}
\label{Example200}
We
continue with the example \ref{Example100}. Assume that $p$ is not isomorphic to $q$ over $\Field$.
That means that $X(\Field)=\emptyset$. Moreover, in the Nisnevich topology, any covering $\{U_i\}$
of $S$ has a section, and so $\prod \Gamma(U_i, X) = \emptyset$.
Thus, the arrow $X \to S$ is not an epimorphism
and $X$ is not a torsor.

\subsubsection{}
Let $X \in \EEE$. Since morphisms in
toposes posses a mono-epi decomposition, one has a well-defined
notion of the image of
any morphism
and we can set
$$
\Supp X = \Img p_X,
$$
where $p_X: X \to 1$ is
the unique map to the final object.
The object $\Supp X$ is called the support of $X$.

Let $P$ be
a formal $G$-torsor.
We can restrict $G$ and $P$ to
$\Supp P$.
In other words,
we
consider
$G$ and $P$
in the topos
$\EEE/\Supp P$. Here $P$ becomes a torsor.
Below we
show that $\Supp P$ is a very interesting invariant of $P$.


\subsubsection{}
In the category of sets the classificator of $G$-torsors on the pre-homotopic level is represented by the topos of $G$-sets.
In a more complicated topos the classificator of torsors
should be represented by something
like an internal topos.
Then the torsor $P$ would induce an arrow
from
$\Supp P$ to this classificator.
It would be exactly the arrow
we need. But, at this stage, we would like
to avoid such constructions. Hopefully, this is possible
since in \cite{MV} the classificator
was constructed on the homotopic level.

\subsection{Torsor classificators
on $\mathcal H_s$-level}
\label{TorsorsClassificationOnHsLevelSS}
The homotopic category in the algebro-geometric context was introduced by F.\,Morel and V.\,Voevodsky
in \cite{MV}. Among other things, their results permit to describe torsors as homotopy classes of maps
to the classificator (which they produced). These results form an indispensable part of our approach as
well.

The homotopic category $\hs(\TTT)$ is obtained from the category of spaces
$\Spc$ via localization with respect to the class of weak equivalences $W_s$
(\cite[Def 2.1.2, p. 48]{MV}).
A morphism of spaces $f:U\row V$ belongs to $W_s$ if and only if
for each point $e$ of the generalized
space (that is, {\it topos})
$\EEE ={\Shv}/\TTT$
the morphism of fibers
$f_e:U_e\row V_e$ is a weak equivalence of simplicial sets.
Here the point $e$ of the topos
$\EEE$,
that is,
a morphism
from
the usual point
$\sbt$
to
$\EEE$,
by definition, is represented by the
pullback functor
$e^*:{\Shv}/\TTT\row {\Shv}/\sbt$, where the
latter category is
the category of sets.
In
this paper
we
need
no
explicit description of points of the generalized space
${\Shv}(\smF{S})_{Nis}$.
We only mention
that with each choice of a scheme $X\in\smF{S}$ together with a point $x\in X$ one can associate a
point $e$ of the topos
$\EEE$. Then $e^*(F)$ is the fiber of $F$ in the hensilization of $x$.
As far as we
understand
this construction gives
essentially
all
the
points of
$\EEE$.

The class $W_s$ is a part of the simplicial model structure $(C_s, W_s, F_s)$ (see \cite[p. 48]{MV}).
The cofibrations, that is, the elements of $C_s$,
are just embeddings. Thus, any object of $\Spc$ is cofibrant.
The fibrations
are defined via the lifting property
w. r. to
the
acyclic cofibrations.

\subsubsection{}
\label{HotModelForSupportSS}
With each smooth $S$-scheme $X$ one can associate
a
space
$EX$
(see \cite[p. 9]{VVMilnorConj}
and
\cite[Exa 4.1.11]{MV}).
By definition $(EX)_n=X^{n+1}$ (products over $S$), with faces and degeneration maps given
by partial projections and partial diagonals.
In other words, we connect by a segment (even by two)
each pair of point, glue up all triangles, etc.

\subsubsection{}
\label{sssHsCl1}
An
$\hs$-classificator of $G$-torsors
$\bzar{G}$
is
constructed by F.\,Morel and V.\,Voevodsky in \cite{MV}.
The space $\bzar{G}$
is
introduced together with
an
isomorphism of functors
$B_\bullet \mapsto P(B_\bullet,G)$ and $B_\bullet \mapsto \hs(B_\bullet, \bzar{G})$,
where $P(B_\bullet,G)$
is the set of isomorphism classes of torsors over $B_\bullet\in \Spc$.
For $U\in \TTT$, by definition, $\bzar{G}(U)$
is the nerve of the category having one object and the
group $G(U)$ of arrows.
Thus, the space $\bzar{G}$ is represented by the simplicial scheme with
$\bzar{G}_{n+1}=G^n$
and
the
standard faces and degenerations maps.
To complete the description of the classificator it remains to specify the morphism of functors
$P(B_\bullet,G) \to \hs(B_\bullet, \bzar{G})$.
For this one needs the space $EG$ (see \ref{HotModelForSupportSS})
as well as the morphism of spaces $\ezar{G} \to \bzar{G}$,
$$
(g_0,\dots,g_n)\mapsto (g_n^{-1}g_{n-1}, \dots, g_2^{-1}g_1,g_1^{-1}g_0).
$$
This map identifies the spaces $G\backslash \ezar{G}$ and $\bzar{G}$.
In \cite[proof of Lemma 4.1.12]{MV}, to each torsor $P_\bullet$ over $B_\bullet$, a hat
$B_\bullet \xleftarrow{p}\tilde{P_\bullet}\xrightarrow{f}\bzar{G}$ is assigned
where $\tilde{P_\bullet}= (\ezar{G} \times_G P_\bullet)$.
The correspondence $P_\bullet \mapsto f\circ p^{-1}$ provides the needed $\hs$-arrow
$B_\bullet \to \bzar{G}$.

\begin{thm}{\rm (F.\,Morel, V.\,Voevodsky, \cite[Proposition 4.1.15]{MV})}
\label{MV-theorem}
The above correspondence defines a bijection
$$
P(B_\bullet,G)\cong\hs(B_\bullet, \bzar{G}).
$$
\end{thm}


\begin{prop}
\label{HotModelForSuppProp}
There exists
a
unique $\Spc$-arrow $p: EX \to \Supp X$. This arrow is a weak equivalence.
Thus, $EX$
is a homotopic model for the support of $X$.
\end{prop}
\begin{proof}
The existence
is the consequence of the locality of the product
($\Supp(X\times Y) = (\Supp X)\times (\Supp Y)$),
while
the uniqueness follows from the
injectivity of the
embedding
$\Supp X \subset 1$.
To check that
$p\in W_s$ it is sufficient to consider the fibers, where everything is reduced to the fact that, for
a non-empty set $X$, the simplicial set $EX$ is weakly equivalent to the point.
\end{proof}

\subsubsection{}
\label{ChechSpaceAndCosceleton} Notice, that $E=\cosk_n$ for $n=0$.
The functor $\cosk_n$ is the right adjoint to the truncation $i_*$ from the category of simplicial
objects to the category of $n$-bounded simplicial objects and should be denoted $i^!$.
For $n=0$ the truncation has the form $Z_\bullet \mapsto Z_0$, and the adjointness means that
$\Hom(Z_\bullet, EX) = \Hom(Z_0,X)$.

\subsubsection{}
\label{EXIsFiberedSSS}
The space $EX$ is locally fibrant by \cite[Lemma 2.1.15]{MV}, that is, fibrant in each fiber.
Indeed, if $t$ is a point, then $(EX)_t = E(X_t)$ is Kan's set, since there everything is glued up
tautalogically.
More accurate argument uses the adjointness from \ref{ChechSpaceAndCosceleton}.
\subsubsection{}
\label{SFinalInHsSSS}
$S$ is the final object of $\mathcal H_s$.
Indeed, in $\Spc$ each object is cofibrant, while the final one is fibrant. Hence, $\mathcal H_s$-arrows
to $S$ are obtained as a quotient of $\mathcal \Spc$-arrows to $S$, that is, as a quotient of the one-element set.

\subsubsection{}
\label{HsArrowsToLocallyFiberedSpaceSSS}
Let $Y_\bullet$ be locally fibrant. Then each $\mathcal H_s$-arrow $f: T_\bullet \to Y_\bullet$ can be
represented by some $\mathcal H_s$-composition $f=q\circ p^{-1}$ corresponding to a hat
$T_\bullet \xleftarrow{p} \tilde T_\bullet\xrightarrow{q}Y_\bullet$,
where $p$ is a local fibration and a weak equivalence. In such a situation, $f$ depends only on the
classes of $p$ and $q$ w.r.to the $s$-homotopic equivalence.
Moreover, for two $\mathcal H_s$-arrows $T_\bullet \to Y_\bullet$ one can find hats with the common
$\tilde T_\bullet$ and $p$ \cite[Prop. 2.1.13]{MV}.

\subsubsection
{$EX$ is a subobject of $S$ in $\mathcal H_s$}.
\label{EXIsSubPointinsHSSS}
In other words, the unique (see \ref{SFinalInHsSSS}) $\mathcal H_s$-arrow $EX \to S$ is a monomorphism.
One needs to check that, for arbitrary space
$T_\bullet$, the set $\mathcal H_s(T_\bullet, EX)$
consists of at most one element.

Indeed, in the light of (\ref{EXIsFiberedSSS}) and (\ref{HsArrowsToLocallyFiberedSpaceSSS}), it is
sufficient to connect arrows $f,g: T_\bullet \to EX$ by an $s$-homotopy $h$.
But inside $EX$ everything is glued up canonically and $h$ is constructed tautologically.
More accurate argument uses the adjointness from \ref{ChechSpaceAndCosceleton}.

\subsubsection{}
\label{HsArrowsFromEYToEXSSS}
Let $Y$ be a smooth
scheme
over $S$, where $S=\Spec\Field$.
Then the following conditions are equivalent (\cite[Lemma 3.8, Remark after Lemma 3.8]{VVMilnorConj}):
\begin{equation}
\label{DominantProperty}
\text{ for any extension of fields }L/\Field:\quad
Y(L) \neq \emptyset \Rightarrow X(L)\neq \emptyset;
\end{equation}
\begin{equation}
\label{HsArrowsFromEYToEXNonEmptyProperty}
\mathcal H_s(EY, EX)\neq \emptyset;
\end{equation}

For irreducible $Y$, the condition
(\ref{DominantProperty}) is equivalent to: $X(L)\neq \emptyset$ for $L = \Field(Y)$.

\subsubsection{}
\label{HsIsoOfEYToEXSSS}
It follows from (\ref{HsArrowsFromEYToEXSSS}) that the following
conditions are equivalent:
\begin{equation}
\label{IsoProperty}
\text{ for each extension of fields }L/\Field:\quad
Y(L) \neq \emptyset \Leftrightarrow X(L)\neq \emptyset;
\end{equation}
\begin{equation}
\label{HsIsoFromEYToEXNonEmptyProperty}
EY \simeq EX \text{ in }\mathcal H_s.
\end{equation}
In particular, the projection $EX \row S$ is an
 $\mathcal H_s$-isomorphism if and only if $X(\Field)\neq \emptyset$.

\subsubsection{}
\label{ClasMap}
Any algebraic $G$-torsor $P$ induces a $G$-torsor over $\Supp P$ and so
(see \ref{MV-theorem}),
an $\hs$-arrow $i_P: \Supp P \to BG$.
The pair $(\Supp P, i_P)$ is exactly the invariant of $P$ we are looking for. The knowledge of this
invariant permits to recover $P$.

There are explicit $\mathcal H_s$-models for the $\Supp P$ (see \ref{HotModelForSuppProp}) and the classificator, namely, $EP$ and $BG$.
Therefore $i_P$ induces the $\mathcal H_s$-arrow between them.
It appears that this arrow can be constructed explicitly already on the level of spaces.
Namely, consider the map
$f_P: EP \to BG, (x_0, \dots, x_n) \mapsto (x_0x_1^{-1}, x_1x_2^{-1},\dots, x_{n-1}x_n^{-1})$.


\begin{prop} The following $\mathcal H_s$-diagram is commutative
(the arrow $p$ is described in \ref{HotModelForSuppProp}):
$$
\xymatrix{
EP  \ar[rr]^{p}\ar[dr]_{f_P}  &  & \Supp P \ar[dl]^{i_P}\\
  & BG &
}.
$$
\end{prop}
\begin{proof} Due to
Theorem of F.\,Morel
and
V.\,Voevodsky (see \ref{MV-theorem}) it is sufficient to show that
$p^*P \simeq f_P^*EG$. Here $p^*P = EP\times P$, and the needed isomorphism of torsors is
given by:
$EP\times P \to EP \times_{BG} EG$:
$(x_0, \dots, x_n)\times x \mapsto (x_0, \dots, x_n)\times (x_nx^{-1}, \dots, x_0x^{-1})$.
\end{proof}


\begin{prop}
\label{Ghiq}
Let $X$ be a $G$-torsor. Then in $\Spc$ we have: $G\backslash(\ezar{G}\times X)=EX$.
\end{prop}

\begin{proof}
It is easy to see that the map
$G\backslash(\ezar{G}\times X)\row EX$ defined by:
$(g_0,\ldots,g_n,x)\mapsto (g_n^{-1}x,\ldots,g_0^{-1}x)$
is an isomorphism of simplicial schemes.
\end{proof}

Below
$EX$
can be denoted
as $\hiq{X}$ (which corresponds to the notations of \cite{Vo2MC}
and \cite{IMQ}).

\subsection{Some remarks}
\subsubsection{}
To simplify computations it would be desirable to pass from
$\hs$
to the motivic homotopic
category of Morel--Voevodsky
$\hH$.
The category $\hH$ is the localization of $\hs$,
identifying $\aaa^1$ with the point (see \cite{MV}).
For computations in $\hH$ it is useful to keep in mind that
$\hH(X,Y)=\hs(X,Y)$
for $\aaa^1$-local $Y$
\cite[Theorem 2.3.2, p. 86]{MV}.

The $\hs$-type of $S$ is clearly $\aaa^1$-local as the final object of $\hs$ (see \ref{SFinalInHsSSS}).
Therefore, $S$ is also the final object of $\hH$.
The $\hs$-type of $EX$ is also $\aaa^1$-local, as a subobject of a local one
(see \ref{EXIsSubPointinsHSSS}).

Moreover, $EX$ is a subobject of $S$ in
$\hH$, that is, the unique
$\hH$-arrow $EX \to S$ is a monomorphism.
Indeed, for arbitrary
$Z_\bullet$, the set
$\hH(Z_\bullet, EX)$
consists of at most one element.
This follows from the $\aaa^1$-locality of $EX$ and the fact that $EX$ is a subobject of $S$ in
$\hs$ (see \ref{EXIsSubPointinsHSSS}).

In the statements \ref{HsArrowsFromEYToEXSSS} and \ref{HsIsoOfEYToEXSSS} one can replace
$\hs$ by $\hH$.
This follows from the $\aaa^1$-locality of $EX$.

\subsection{Nisnevich and \'etale classifying spaces}

\subsubsection{}
\label{OtherTopologiesSSS}
Everything
said above about the classification of torsors in
the Nisnevich topology,
can be also applied to other topologies, in particular, to the \'etale one.
This way, we get
a
classificator in
$\hs(\TTT)$ for the respective site.
As
a family
of classificators and supports appears, the corresponding notations should
contain references to $\TTT$.

For the \'etale topology, $\betO{G}\in\hs((\smF{S})_{\text{et}})$.
But, in \cite{MV} the notation $\betO{G}$ denotes something else, namely,
$R\pi_*(\betO{G})$, where
$$
\xymatrix {
\hs((\smF{S})_{et}) \ar@<0.5ex>[d]^{{R\pi_*}}\\
\hs((\smF{S})_{Nis}) \ar@<0.5ex>[u]^{L\pi^*=\pi^*} \\
}
$$
is the pair of conjugate functors induced by the morphism of sites
$\pi:(\smF{S})_{et}\row (\smF{S})_{Nis}$.
Therefore, to avoid confusion, we set
$\bet{G} = R\pi_*(\betO{G})$.

Besides, the $\mathcal H_s$-type of $\bet{G}$ is defined in \cite{MV} slightly differently,
namely, as
$R\pi_*(\pi^*\bzar{G})$. The result is the same though, since
$\betO{G} = \pi^*(\bzar{G})$.

\subsubsection{}
\label{NisAndEtBAreDifferentSSS}
The $\hs$-types $\bzar{G}$ and $\bet{G}$ are quite different.
For example, there is only one $\hs$-arrow from $\Spec \Field$ to
$\bzar{G}$, while $\hs$-arrows $\Spec \Field \to \bet{G}$ correspond to
the isomorphism classes of algebraic $G$-torsors. For $G=O(n)$
it is the set of the isomorphism classes of $n$-dimensional
quadratic forms over $\Field$, which is
very non-trivial.

\subsubsection{}
\label{ArrowFromBGToBGEtSSS}
The discussion of \ref{OtherTopologiesSSS} shows that the identity map on $\betO{G}$,
by adjunction,
induces an $\hs$-arrow
$$
\varepsilon: \bzar{G}\to\bet{G}.
$$

Any algebraic $G$-torsor $P$ induces $\hH_s$-arrow $i_P: \Supp P \to \bzar{G}$ (see \ref{ClasMap}). On the other hand it induced an
$\hH_s((\smF{S})_{et})$-arrow
$i_{et,P}: S \to \betO{G}$ since the corresponding etale support coincides with $S$. By adjunction, we get the corresponding $\hH_s$-arrow
$i_{et/Nis,P}: S \to \bet{G}$. The arrows $i_P$ and $i_{et/Nis,P}$ are related by means of $\varepsilon$. Namely, the following $\hH_s$-diagram is commutative.
$$
\xymatrix {
\Supp P \ar[d]_{\subset}\ar[rr]^{i_P} & & \bzar{G}\ar[d]^{\varepsilon}\\
S \ar[rr]^{i_{et/Nis,P}}& & \bet{G} \\
}
$$

We know an $\hH_s$-model for $\Supp P$ (see \ref{ClasMap}). In addition, there are nice geometric $\hH$-models for
$\bet{G}$ (see \ref{ArrowFromBGToGeomBGSSS}).
We are going to relate these models by means of $\varepsilon$. Moreover, we are going to get a similar diagram not for an individual algebraic torsor, but rather for a family of such torsors.
\subsubsection{}
\label{ArrowFromBGToGeomBGSSS}
For any exact representation $\rho: G \to GL(V)$
a geometric model $\bgm{(G, \rho)}$ for the $\hH$-type of $\bet{G}$ is constructed in
\cite[Prop. 4.2.6]{MV}. As a space, that is as an object of $\Spc$,
it is defined as the quotient $G\backslash \egm{(G,\rho)}$, where $\egm{(G,\rho)}$ is an open subscheme of $V^{\infty}$
consisting of the points where the diagonal action of $G$ is free.
It is proved that $\egm{(G,\rho)}$ is $\hH$-contractible and that $\bgm{(G, \rho)}$ is $\hH$-isomorphic to $\bet{G}$.
The choice of the representation $\rho$ will not be important for us, so below we denote the
respective spaces simply as $\egm{G}$ and $\bgm{G}$.

This isomorphism can be described explicitly.
Indeed, the algebraic torsor $\egm{G}$ over $\bgm{G}$ induces a canonical $\hH_s((\smF{S})_{et})$-arrow $i_{et}$
from its base $\bgm{G}$ to the classificator $\betO{G}$. By adjunction, we get the corresponding $\hH_s$-arrow
$i_{et/Nis}: \bgm{G} \to \bet{G}$.
This arrow becomes the $\hH$-isomorphism we look for.

Let $x: S \to \bgm{G}$ be a rational point of the (inductive) scheme $\bgm{G}$.
Denote by $P_x$ the algebraic $G$-torsor
$\pi^{-1}(x)$, where $\pi$ is the projection $\egm{G} \to \bgm{G}$.

\subsubsection{}
\label{???}
To relate the homotopic models and $\varepsilon$ we need more spaces and arrows.
Consider the $\Spc$-diagram
$$
\bzar{G}\stackrel{p}{\low}G\backslash\left(\ezar{G}\times \egm{G}\right)\xrightarrow{\tilde{\eps}_\rho}\bgm{G},
$$
where the action on the middle term is diagonal and the arrows $p$ and $\tilde{\eps}_\rho$ are induced by the projections.
It follows from \cite[Prop. 4.2.3]{MV} that $p$ is an $\hH$-isomorphism.
Set
$$
\bzart{G}=G\backslash(\ezar{G}\times \egm{G}).
$$

Let $\varphi: \cu\row\cv$ be a $\Spc$-arrow
represented by a termwise smooth morphism.
Then we have the well-defined $\Spc$-product
$\cu\times_\cv\times \dots \times_\cv\cu$. Denote it by $(\cu/\cv)^{n}$.
Consider the following simplicial space $sE_\bullet(\cu/\cv)\in \sSpc$, where $sE_n(\cu/\cv) = (\cu/\cv)^{n+1}\in \Spc$, and the face and degeneration
morphisms are partial projections and partial diagonals. Applying the diagonal functor
$\sSpc\row\Spc$ to $sE_\bullet(\cu/\cv)$ we get a space $E_\bullet(\cu/\cv)=\hiq{\cu/\cv}$ together with the structure $\Spc$-arrow $\hiq{\cu/\cv}\row\cv$. Denote this map as ${\cal X}(\cu\row\cv)$.

\begin{prop}
\label{etzar}
$$
{\cal X}\left(\egm{G}\stackrel{\pi}{\row}\bgm{G}\right)=
\bzart{G}\stackrel{\tilde{\eps}_{\rho}}{\row}\bgm{G}.
$$
\end{prop}

\begin{proof}
It follows from the fact that $(\egm{G}/\bgm{G})^{r+1}=\egm{G}\times G^{r}$ and
faces and degeneration maps are as in $\bzart{G}$ (which on the fibers can be seen from
Prop. \ref{Ghiq}).
\end{proof}

In particular, the fiber $(\tilde{\eps}_{\rho})^{-1}(x)$ over the rational point $x\in \bgm{G}$
is $\hiq{P}=EP_x$ (see \ref{ArrowFromBGToGeomBGSSS} for the notation $P_x$), and
the composition $E(P_x)\hookrightarrow\bzart{G}\stackrel{p}{\lrow}\bzar{G}$ is exactly the map
$i_P$ from \ref{ClasMap}. Abusing notations somewhat we will use the name "Nisnevich-\'etale fiber"
for this classifying map.

\subsection{Groups and groupoids}

Let $G$ be a linear algebraic group over a field $\Field$.
An anti-automorphism: $g\mapsto g^{-1}$ of $G$
identifies the set of left and right $G$-torsors: $X\mapsto X^{-1}$.

If $X$ is a left algebraic $G$-torsor, the functor
$U\mapsto \Aut_G(X\times U\row U)$, by \'etale descent is
represented by some group scheme $G_X$ over $\Field$, and $X$ has
a natural structure of the right $G_X$-torsor.
We get a {\it torsor triple}, i. e. the triple $(G,X,H)$,
where $G$ and $H$ are group schemes, and $X$ is a left $G$-torsor, and a right $H$-torsor,
and these structures commute. Such triples can be composed:
$(G,X,H)\circ (H,Y,K):=(G,X\times_H Y,K)$, and inverted: $(G,X,H)^{-1}:=(H,X^{-1},G)$, and form a
{\it groupoid}. In particular, $(G,X,H)\circ (G,X,H)^{-1}=id_G=(G,G,G)$.
Thus, torsor triples are just morphisms of our groupoid.

The following statement shows how the classifying spaces corresponding to two different groups from
the same groupoid are related.

\begin{prop}
\label{GYH}
For any torsor triple $(G,Y,H)$ there is a natural $\hs$-identification
$$
\xymatrix @-0.7pc{
\hiq{Y}\times\bzar{H}^R \ar @{=}[rr] \ar @{->}[rd] & &\bzar{G}^L\times\hiq{Y} \ar @{->}[ld]\\
& \hiq{Y} &
}
$$
such that the natural projections
$\hiq{Y}\times\bzar{H}^R\row\bzar{H}^R$ and $\bzar{G}^L\times\hiq{Y}\row\bzar{G}^L$
map "the other" copy of $\hiq{Y}$ to the Nisnevich-\'etale fibers over $[Y]$ (the map from \ref{ClasMap}).
\end{prop}

\begin{proof}
Consider $\ezar{G}\times Y\times\ezar{H}$ with the left $G$-action on the first two factors and
the right $H$-action on the last two. From Proposition \ref{Ghiq}, in $\Spc$ we have an identification:
$$
G\backslash(\ezar{G}\times Y\times\ezar{H})=\hiq{Y}\times\ezar{H}\,\,\,\text{ and }\,\,\,
(\ezar{G}\times Y\times\ezar{H})/H=\ezar{G}\times\hiq{Y}
$$
with the standard right $H$ and left
$G$-action, respectively. And $G\backslash(\ezar{G}\times\hiq{Y})$ and
$(\hiq{Y}\times\ezar{H})/H$ are just homotopic quotients of $\hiq{Y}$ by these actions. But, by
\ref{EXIsSubPointinsHSSS},
$$
\Hom_{\hs(\Field)}(G\times\hiq{Y},\hiq{Y})=*=
\Hom_{\hs(\Field)}(\hiq{Y}\times H,\hiq{Y}).
$$
Thus, our actions on $\hiq{Y}$ are homotopically trivial.
Hence, $G\backslash(\ezar{G}\times\hiq{Y})=\bzar{G}^L\times\hiq{Y}$ and
$(\hiq{Y}\times\ezar{H})/H=\hiq{Y}\times\bzar{H}^R$ in $\hs(\Field)$.
Thus, we get an identification in $\hs(\Field)$:
$$
\hiq{Y}\times\bzar{H}^R=G\backslash(\ezar{G}\times Y\times\ezar{H})/H=\bzar{G}^L\times\hiq{Y}.
$$
Since $\hiq{Y}$ is a subobject of the final object $\sbt$ of $\hs(\Field)$,
this identification is over $\hiq{Y}$.

Consider the $G-H$-equivariant projection:
$\ezar{G}\times Y\times\ezar{H}\row\ezar{G}\times\ezar{H}$ giving the map:
$G\backslash(\ezar{G}\times Y\times\ezar{H})/H\row\bzar{G}^L\times\bzar{H}^R$.
Since the maps $G\backslash(\ezar{G}\times Y\times H)/H\row (\hiq{Y}\times\bullet)$ and
$G\backslash(G\times Y\times\ezar{H})/H\row (\bullet\times\hiq{Y})$ are isomorphisms
(here $\bullet$ is the only homotopic rational point on $\bzar{H}^R$ and $\bzar{G}^L$), we see that
the map $(\hiq{Y}\times\bullet)\row\bzar{G}^L$ is induced by the $G$-equivariant map
$\ezar{G}\times Y\row\ezar{G}$, while the map $(\bullet\times\hiq{Y})\row\bzar{H}^R$ is
induced by the $H$-equivariant map $Y\times\ezar{H}\row\ezar{H}$. By Lemma \ref{Ghiq}
these are exactly the Nisnevich-\'etale fibers over $[Y]$.
\end{proof}


\subsubsection{}
Note, that left and right classifying spaces $\bzar{G}^L$ and $\bzar{G}^R$ are canonically isomorphic.

In the above situation, $\bzar{G}$ is not isomorphic to $\bzar{H}$, in general.
The situation with the \'etale classifying spaces is different.

\begin{prop}
\label{etGYH}
For any torsor triple $(G,Y,H)$ there is canonical $\hH_s$-isomorphism
$$
\bet{H}\stackrel{\theta_Y}{\row}\bet{G},
$$
which acts on homotopic rational points by $[X]\mapsto [Y\times_H X]$.
\end{prop}

\begin{proof}
By Proposition \ref{GYH}, we have a natural identification
$\hiq{Y}\times\bzar{H}\stackrel{\cong}{\row}\bzar{G}\times\hiq{Y}$.
Since $\bet{G}=R\pi_*\circ\pi^*\bzar{G}$, and
$\pi^*(\hiq{Y})=\bullet_{et}$  and $\pi^*$ respects products, we get an identification
$\bet{H}\stackrel{\theta_Y}{\row}\bet{G}$.

Since the fibers $\hiq{X}\row\bzar{H}$ and $\hiq{X\times_H Y}\row\bzar{G}$ are given by
the $H$, respectively, $G$-equivariant maps $\ezar{H}\times X\row\ezar{H}$ and
$\ezar{G}\times (X\times_H Y)\row\ezar{G}$, and
$(Y\times\ezar{H}\times X)/H=(Y\times_H X)\times\hiq{Y}$ we get that
$[X]\mapsto [Y\times_H X]$.
\end{proof}

The above map does not preserve a base-point (given by a trivial torsor).

\subsection{Some invariants of torsors}

\subsubsection{}
As soon as we have got an arrow $\Supp P\row BG$, we can apply cohomology theories to it. For example, to such a theory $A$ we can associate an invariant:
$$
\kker_A(P):=\kker[A(BG)\row A(\Supp P)].
$$
The task is to determine which information on $P$ this invariant carries, and to describe the
possible values of it.


\subsubsection{}
For an algebraic subgroup $H \subset G$ and an algebraic $H$-torsor $Q$ we have an algebraic $G$-torsor $P$ defined by
$P=G\times_H Q$.
This fits into the commutative $\mathcal H_s$-diagram
$$
\xymatrix{
\Supp Q \ar[r] \ar[d] & BH\ar[d] \\
\Supp P \ar[r] & BG
}.
$$
It gives the commutative diagram
\begin{equation}
\label{YXalpha}
\xymatrix{
A(BG) \ar[r]  \ar[d]   & A(\Supp P)  \ar[d] \\
A(BH)  \ar[r]                & A(\Supp Q).
}
\end{equation}
This enables one to get information about $\kker_A(P)$ from $\kker_A(Q)$ for induced torsors.

\section{Subtle Stiefel--Whitney Classes}
\label{CohomologyOfBOS}
We would like to apply the technique developed above to the case of an
orthogonal group with the aim of classifying quadratic forms.
The first step is to choose an appropriate cohomology theory $A$ and
to compute the cohomology ring of the classificator.
Our theory will be motivic cohomology $\hm^{*,*'}$.
From now on let $\charac \Field\neq 2$.
Set
$H=\hm^{*,*'}(\Spec \Field,\zz/2)$.
By the result of V.\,Voevodsky \cite{VVMilnorConj},
$H=K^M_*(\Field)/2[\tau]$, where $\tau$ is the only non-zero
element of degree $(1)[0]$.

\subsection{Motivic cohomology of $BO(n)$}
\label{MotCohBOn}
Everywhere below $\sqrt{-1}\in \Field$.
For $n = 1, 2,\dots $ the form
$q_n = x_1^2 + \dots + x_n^2$
will be called the standard one, and the respective automorphism group will be denoted $O(n)$.
Then $n$-dimensional quadratic forms will correspond to the (left) $O(n)$-torsors via the rule:
$q\leftrightarrow X_q=\Iso(q\row q_n)$.
Let
$$
w_i\in\hm^{i,i}(\bet{O(n)},\zz/2), \quad c_i\in\hm^{2i,i}(\bet{O(n)},\zz/2)
$$
be the Stiefel--Whitney
and Chern classes. Slightly abusing the notation
denote $\eps^*$ of them also as $w_i$ and $c_i$, where
$\eps: \bzar{O(n)}\row\bet{O(n)}$ is the arrow from \ref{ArrowFromBGToBGEtSSS}.


\begin{thm}
\label{u}
There is a unique set $u_1,\dots, u_n$ of classes in the motivic $\zz/2$-cohomology for $BO(n)$
such that $\ddeg u_i=([i/2])[i]$,
$$
\hm^{*,*'}(\bzar{O(n)},\zz/2)=H[u_1,\ldots,u_n], \quad w_i=u_ i\tau^{[(i + 1)/2]},
\text{ and }c_i=u_i^2\tau^{\delta(i)}.
$$
Here $\delta(i)$ is the indicator of oddness, that is $\delta=0$, for $i$ even, and $\delta=1$, for $i$ odd.
\end{thm}

The Theorem will be proven in \ref{ProofUTheoremSSS}. We start with some preliminary observations.

Let $p$, $q$ be quadratic forms. Denote as $I(p,q)$ the (smooth) variety of isometric embeddings
from $p$ to $q$. For example, $A_q=I(q_1,q)$
is the affine quadric $q=1$.

\begin{prop}
\label{Affquad}
Let $q$ be a quadratic form of dimension $n$. Then
$$
O(n-1)\backslash X_q\cong A_q,\hspace{3mm} \text{in particular,}\hspace{3mm}
O(n-1)\backslash O(n)\cong I(q_1,q_n).
$$
\end{prop}

\begin{proof}
An isomorphism is given by: $x\mapsto x^{-1}(0,...,0,1)$.
\end{proof}


\begin{prop}
\label{motAff}
Let $M:\mathcal H\row\dm{\Field}$ be the motivic functor, $A=I(q_1,q_n)$. Then
$$
M(A)=\zz\oplus\zz([n/2])[n-1].
$$
\end{prop}

\begin{proof} Consider $r=\la -1\ra \perp q_n$ with the respective quadric $R$. Then
$Q\subset R$ is a codimension one subquadric with complement $A$.
In $\dm{\Field}$ we have Gysin's exact triangle:
$$
M(A) \xrightarrow{} M(R) \xrightarrow{j} M(Q)(1)[2]\xrightarrow{}M(A)[1].
$$
Since the quadrics $R$ and $Q$ are split, $\op{Cone}[-1](j)=\zz\oplus\zz([n/2])[n-1]$ (we use
the fact that $\sqrt{-1}\in \Field$.).
\end{proof}

\subsubsection{}
\label{MotOverSimplicialSchemasSSS}
Let $Y_\bullet \in \Spc$, $R$ be a commutative ring. The category
$\dm{Y_\bullet}$ is introduced in \cite{VoMSS}. The notation there is slightly different (minus is a subscript),
but we prefer to denote it the above way, since it reflects the fact that the cohomological indices of
the non-trivial terms of a complex are bounded from above. This is coherent with the derived category
notations.
We need the category $\dmR{Y_\bullet}{R}$. This category is not introduced in \cite{VoMSS}, but it is
mentioned there in \S 7 that all the results can be extended to the case with coefficients. So, we will use the $R$-analogues referring to the respective $\zz$-formulations.

By definition, $\dmR{Y_\bullet}{R}$ is the localization of $D(Y_\bullet, R)$ \cite[Def 4.2]{VoMSS},
that is (\cite[after Lemma 2.3]{VoMSS}), of the derived category $D^{-}(\PST(Y_\bullet, R))$,
where $\PST(Y_\bullet, R)$ is the category of the presheaves of $R$-modules with transfers on $Y_\bullet$.
The objects of $\dmR{Y_\bullet}{R}$ are complexes
$\dots \to C_{i+1} \to C_i \to \dots \to 0 \to 0 \to \dots$,
where $C_j\in \PST(Y_\bullet, R)$.

The category $\PST(Y_\bullet, R)$ is defined as in \cite[Def 2.1]{VoMSS}) with the replacement of Abelian groups  by $R$-modules.
An object $K\in\PST(Y_\bullet, R)$ is represented by the system $\{K_n,f_{\theta}\}$, where
$K_n\in \PST(Y_n, R)$ and $f_{\theta}:(Y_{\theta}^*)(K_n)\row K_m$, for
$\theta:[n]\row [m]$ is a coherent system of arrows.

We need also functors
$$
r_i^*: \dmR{Y_\bullet}{R} \to \dmR{Y_i}{R}, \quad
r_i^*(N) = N_i.
$$

Following \cite{VoMSS}, for a space $Y_{\bullet}$ we denote as $CC(Y_{\bullet})$
the simplicial set, where $CC$ is the functor commuting with the coproducts and sending
a connected scheme to the point.


\begin{prop}
\label{gradTate}
Suppose that $H^1(CC(Y_\bullet), R^{\times})=0$.
Let $M = T(u)[v] \in \dmR{S}{R}$ be the Tate-motive.
Let $N\in\dmR{Y_\bullet}{R}$ be such a motive
that its graded components $N_i\in\dmR{Y_i}{R}$ are isomorphic to $M$ and all the structure maps $N_{\theta}:LY_{\theta}^*(N_i)\row N_j$
are isomorphisms (here $\theta:[i]\row [j]$ is a simplicial map).
Then $N$ is isomorphic to $M$.
\end{prop}

\begin{proof}
Consider the conjugate pair of functors (see \cite{VoMSS}):
$$
Lc_{\#}:\dmR{Y_{\bullet}}{R}\row\dmR{S}{R}\hspace{3mm}\text{and}\hspace{3mm}
c^*:\dmR{S}{R}\row\dmR{Y_{\bullet}}{R}.
$$
Objects $N$ with the property that all the structural maps $N_{\theta}$ are isomorphisms form a full localizing subcategory
$\dmcohR{Y_{\bullet}}{R}$ of "coherent objects" of $\dmR{Y_{\bullet}}{R}$.
By Proposition \ref{cohPost} below, an  object $N$ of this subcategory has a functorial (increasing) filtration
$(N)_{\leq n}$ with graded pieces $(N)_n\cong Lr_{n,\#}r_n^*(N)[n]$ (here $(N)_n=\op{Cone}((N)_{\leq n-1}\row N_{\leq n})$) and
this filtration "converges" in the sense that the sequence
$$
\oplus_n(N)_{\leq n}\stackrel{id-i}{\lrow}\oplus_n(N)_{\leq n}\row N
$$
extends to a distinguished triangle
(here $i:(N)_{\leq n}\row (N)_{\leq n+1}$ is the map from the filtration data).
Applying the functor $Lc_{\#}$, we get a functorial in $N$ filtration $(Lc_{\#}N)_{\leq n}$ on $Lc_{\#}N$ with graded pieces
$(Lc_{\#}N)_n\cong Lc_{n,\#}r_n^*(N)[n]$ and the property that $Lc_{\#}N$ is isomorphic to
$\op{Cone}(\oplus_n(Lc_{\#}N)_{\leq n}\stackrel{id-i}{\lrow}\oplus_n(Lc_{\#}N)_{\leq n})$.

In our situation, $(Lc_{\#}N)_n\cong M(Y_n)(u)[v+n]$. Denoting as $(Lc_{\#}N)_{>n}$ the cone
$\op{Cone}((Lc_{\#}N)_{\leq n}\row Lc_{\#}N)$ and as $(Lc_{\#}N)_{m\geq *>n}$ the cone
$\op{Cone}((Lc_{\#}N)_{\leq n}\row (Lc_{\#}N)_{\leq m})$, we get the isomorphism
$(Lc_{\#}N)_{>n}\cong\op{Cone}(\oplus_{m>n}(Lc_{\#}N)_{m\geq *>n}\stackrel{id-i}{\lrow}\oplus_{m>n}(Lc_{\#}N)_{m\geq *>n})$, for $m\geq n$.
From this, using the fact that $(Lc_{\#}N)_{m\geq *>n}$ is an extension of $(Lc_{\#}N)_k$, for $m\geq k>n$,
and the fact that for a smooth variety $Z$ over $S$, $\Hom_{\dmR{S}{R}}(Z,T[j])=0$, for $j<0$, we get:
\begin{equation*}
\begin{split}
&\Hom_{\dmR{S}{R}}((Lc_{\#}N)_{>0},T(u)[v])=0;\\
&\Hom_{\dmR{S}{R}}((Lc_{\#}N)_{>1},T(u)[v+k])=0,\hspace{2mm}\text{for}\,k=0,1.
\end{split}
\end{equation*}
Then we have an exact sequence:
$$
0\row\Hom(Lc_{\#}(N),T(u)[v])\row\Hom(Lc_{\#,0}(N_0),T(u)[v])\row\Hom(Lc_{\#,1}(N_1),T(u)[v]).
$$
And the same is true with $N$ replaced by $M$ (since this object is also "coherent").

Since each $N_n$ can be identified with the $(T(u)[v])_n$,  we can identify the group\\
$\Hom_{\dmR{S}{R}}(Lc_{\#,n}(N_n),T(u)[v])$ with $\Hom_{\dmR{S}{R}}(M(Y_n),T)$ which is a free $R$-module of
the rank equal to the number of connected components of $Y_n$. Let us choose a basis (supported on connected components).
For each $\theta:[i]\row [j] \in Mor(\Delta)$, we
have the natural map $N_{\theta}:(LY_{\theta}^*)(N_i)\row N_j$.
So, the choice of $\theta$ and of connected component of $Y_j$ gives us an element of
$R^{\times}$ (as an automorphism of a Tate-motive on a connected variety). Thus, the obstruction to
identifying the $\Hom(Lc_{\#}(N),T(u)[v])$ with $\Hom(Lc_{\#}(M),T(u)[v])$ lies in
$\Hom_{gr}(\pi_1(CC(Y_{\bullet})),R^{\times})$.
Since $H^1(CC(Y_\bullet), R^{\times})=0$, the above exact sequences for $M$ and $N$ can be identified. In particular,
$\Hom(Lc_{\#}(N),T(u)[v])$ can be identified with $\Hom(Lc_{\#}(M),T(u)[v])$,
and we get a map $\psi:Lc_{\#}(N)\row T(u)[v]$ such that the composition of it with the natural map $(Lc_{\#}(N))_0\row Lc_{\#}(N)$
gives an element in $\Hom_{\dmR{S}{R}}(M(Y_0),T)$ whose coordinates corresponding to all the connected components of $Y_0$
are invertible. By conjugation, it gives us the map
$\ffi:N\row T(u)[v]$ in $\dmR{Y_{\bullet}}{R}$. From the construction, $\ffi_0$ is an isomorphism. Since all the structure maps
$N_{\theta}$ and $T(u)[v]_{\theta}$ are isomorphisms, all the maps $\ffi_i$ are isomorphisms as well, and so is $\ffi$ by
\cite[Lemma 4.4]{VoMSS}.
\end{proof}

In particular, the above result works if $R=\zz/2$, or if $R$ is a field of characteristic $p$
and $\pi_1(CC(Y_{\bullet}))$ is a $p$-group.

\begin{prop}
\label{coh-cond}
Let $N$ be an object of $\dmR{Y_\bullet}{R}$. Then the following conditions are equivalent:
\begin{itemize}
\item[$1)$ ] For every $\theta:[i]\row[j]$, the structure map $N_{\theta}:LY_{\theta}^*(N_i)\row N_j$ is an isomorphism.
\item[$2)$ ] For every $i$, the natural map $Lr_{i,\#}r_i^*(N)\stackrel{\mu_i}{\lrow} N\stackrel{L}{\otimes} Lr_{i,\#}r_i^*(\laa)$
is an isomorphism (here $\laa$ is the unit object of our tenzor triangulated category).
\end{itemize}
\end{prop}

\begin{proof}
The $j$-th component of $Lr_{i,\#}r_i^*(N)$ can be identified with the $\oplus_{\theta:[i]\row[j]}LY_{\theta}^*(N_i)$, while
the $j$-th component of $N\stackrel{L}{\otimes} Lr_{i,\#}r_i^*(\laa)$ can identified with the
$\oplus_{\theta:[i]\row[j]}N_j$ and $r_j^*(\mu_i)$ with the $\oplus_{\theta:[i]\row[j]}N_{\theta}$.
Since the morphism in $\dmR{Y_\bullet}{R}$ is an isomorphism if and only if all of its graded components are
(\cite[Lemma 4.4]{VoMSS}), we obtain the result.
\end{proof}

\begin{defi}
\label{coh-cat}
Denote as $\dmcohR{Y_{\bullet}}{R}$ the full localizing subcategory of $\dmR{Y_{\bullet}}{R}$ consisting of objects satisfying the
conditions of Proposition \ref{coh-cond}.
\end{defi}

\begin{prop}
\label{cohPost}
On an object $N$ of $\dmcohR{Y_{\bullet}}{R}$ we have a functorial increasing filtration $(N)_{\leq n}$ with the graded pieces
$(N)_n\cong Lr_{n,\#}r_n^*(N)[n]$. This filtration "converges" in the sense that the sequence
$$
\oplus_n(N)_{\leq n}\stackrel{id-i}{\lrow}\oplus_n(N)_{\leq n}\row N
$$
extends to a distinguished triangle.
\end{prop}

\begin{proof}
It follows from \cite[Lemma 3.9]{VoMSS} that the unit object $\laa$ of our tenzor triangulated category has the respective filtration.
Applying the $N\stackrel{L}{\otimes}$ functor we get the functorial "converging" filtration on $N$ with graded pieces isomorphic to
$N\stackrel{L}{\otimes}Lr_{n,\#}r_n^*(\laa)[n]$. Since $N$ belongs to $\dmcohR{Y_{\bullet}}{R}$, these graded pieces can be
(functorially) identified with $Lr_{n,\#}r_n^*(N)[n]$.
\end{proof}

\subsubsection{Proof of Theorem \ref{u}.}
\label{ProofUTheoremSSS}
Set $u_0=1$ and use induction on $n$.
As the base take $n=0$.
For $n>0$, consider the transition $(n-1) \to n$.
Identifying $q_n=q_{n-1}\perp\la 1\ra$, we get an embedding
$O(n-1)\hookrightarrow O(n)$.
Denote as $\btzar{O(n-1)}$ (respectively, $\bttzar{O(n-1)}$) the quotient
$O(n-1)\backslash(\ezar{O(n-1)}\times\ezar{O(n)})$
with the diagonal action (respectively, $O(n-1)\backslash\ezar{O(n)}$) in $\Spc$.
Then we have natural maps in $\Spc$
(see the very beginning of \ref{TorsorsClassificationOnHsLevelSS}):
$$
\bzar{O(n-1)}\stackrel{\ffi}{\llow}\btzar{O(n-1)}
\stackrel{\psi}{\lrow}\bttzar{O(n-1)}
$$

Then $\ffi$ is an isomorphism in $\hsF{\Field}$. Indeed, over each graded component of $\bzar{O(n-1)}$ this is a trivial fibration
with the contractible fiber $\ezar{O(n)}$. In particular, it induces an
isomorphism on motivic cohomology.
On the other hand, the fiber of $\psi$ over a point of a graded component of $\bttzar{O(n-1)}$ is a \v{C}ech simplicial scheme $\hiq{X}$,
corresponding to such $O(n-1)$-tosor $X$ whose extension to $O(n)$ is split, so $X$ itself is split as an $O(n-1)$-torsor. Hence, $\psi$ is
an isomorphism in $\hsF{\Field}$ as well, and induces an isomorphism on motivic cohomology.

We have a natural fibration $\bttzar{O(n-1)}\row\bzar{O(n)}=O(n)\backslash\ezar{O(n)}$
with fibers $O(n-1)\backslash O(n)$, which is trivial over the simplicial components.

Let $M(\bttzar{O(n-1)}\row\bzar{O(n)})\stackrel{g}{\row}T$ be the natural projection in
$\dmDVA{\bzar{O(n)}}$. Then it follows from Lemmas \ref{Affquad}, \ref{motAff}, and \ref{gradTate}
(taking into account that our coefficients are $\zz/2$)
that $\op{Cone}[-1](g)=T([n/2])[n-1]$. Recalling the fact about $\psi$ above, we
obtain an exact triangle in $\dmR{\bzar{O(n)}}{\zz/2}$:
$$
\xymatrix @-1.7pc{
& M(\bttzar{O(n-1)}\row\bzar{O(n)}) \ar @{->}[rddd]^(0.5){g_n} &  \\
& & \\
& \star & \\
T([n/2])[n] \ar @{->}[ruuu]_(0.5){[1]}^(0.5){h_n} &  & T
\ar @{->}[ll]^(0.5){f_n}.
}
$$
Using the property of $\ffi$ we get an induced diagram:
$$
\xymatrix @-1.7pc{
& \hm^{*,*'}(\bzar{O(n-1)},\zz/2) \ar @{->}[lddd]^(0.5){[1]}_(0.5){h_n^*} &  \\
& & \\
& \star & \\
\hm^{*-n,*'-[n/2]}(\bzar{O(n)},\zz/2) \ar @{->}[rr]_(0.5){f_n^*} &  &
\hm^{*,*'}(\bzar{O(n)},\zz/2)
\ar @{->}[luuu]_(0.5){g_n^*},
}
$$
where $f_n^*$ is multiplication by some non-zero
$u_n\in\hm^{n,[n/2]}(\bzar{O(n)},\zz/2)$.
By induction,
$\hm^{*,*'}(\bzar{O(n-1)},\zz/2)$ is generated by $u_1,\ldots,u_{n-1}$.
Since $\hm^{?,<0}(\bzar{O(n)},\zz/2)=0$, and $f_n^*$ is injective on $\hm^{0,0}(\bzar{O(n)},\zz/2)=\zz/2$,
we get that $h_n^*(u_i)=0$, for $i=0,\ldots,n-1$.
In particular, $u_i,i=1,\ldots n-1$ can be uniquely lifted to $\hm^{*,*'}(\bzar{O(n)},\zz/2)$.
Since $g_n^*$ is a ring homomorphism, it is surjective.
Hence, $h_n^*=0$ and $\hm^{*,*'}(\bzar{O(n)},\zz/2)=H[u_1,\ldots,u_n]$.

Let us compare $u_i$'s with $\omega_i$'s and $c_i$'s.
Start with $n=1$ case:
$\bzar{O(1)}=\bzar{(\zz/2)}$. Let $\bzar{O(1)}\stackrel{\eps}{\row}\bet{O(1)}$ be
the Nisnevich-\'etale $\hs$-map, and $\omega_1\in\hm^{1,1}(\bet{O(1)},\zz/2)$,
$c_1\in\hm^{2,1}(\bet{O(1)},\zz/2)$ be the usual Stiefel--Whitney and Chern classes.
Then $\eps^*(\omega_1)=\tau\cdot u_1+\{a\}$, for some $\{a\}\in K^M_1(\Field)/2$.
But the only homotopic rational point $\bullet$ of $\bzar{O(1)}$ is mapped to the fixed rational
point of $\bet{O(1)}$ (corresponding to the trivial torsor $\Iso(\la 1\ra\row\la 1\ra)$), so the restriction
of $\omega_1$ to this point is zero. On the other hand, it is equal to $\{a\}$. Thus,
$\eps^*(\omega_1)=\tau\cdot u_1$.
Analogously, $\eps^*(c_1)=\tau\cdot u_1^2+\{b\}\cdot u_1$, for some $\{b\}\in K^M_1(\Field)/2$.
Since $\tau\cdot c_1=\omega_1^2$ (as $-1$ is a square in $\Field$),
we obtain that $\{b\}=0$ and $\eps^*(c_1)=\tau\cdot u_1^2$.

In particular, this shows that $\omega_i$ restricts non-trivially to
$\hm^{*,*'}(\bzar{\left(\times_{i=1}^nO(1)\right)},\zz/2)$.
Comparing the Nisnevich and \'etale classificators for $O(i)$ and $O(i-1)$
we obtain that $\eps^*(\omega_i)$ is divisible by $u_i$. By degree consideration,
we must have $\eps^*(\omega_i)=\tau^{[\frac{i+1}{2}]}\cdot u_i$. But looking at the
restriction to $\hm^{*,*'}(\bzar{\left(\times_{i=1}^nO(1)\right)},\zz/2)$ we also obtain:

\begin{prop}
\label{i111n}
Let $\left(\times_{j=1}^nO(1)\right)\stackrel{\delta}{\row} O(n)$ be the standard embedding.
Then
$$
\delta^*(u_i)=\tau^{[i/2]}\cdot\sigma_i(x_1,\ldots,x_n),
$$
where $x_j$ is $u_1$ from the $j$-th
component, and $\sigma_i$ is the $i$-th elementary symmetric function.
In particular, the map
$\delta^*:\hm^{*,*'}(\bzar{O(n)},\zz/2)\row\hm^{*,*'}(\bzar{\left(\times_{i=1}^nO(1)\right)},\zz/2)$
is injective.
\end{prop}

\begin{proof}
The formula is clear from the description of $\delta^*\eps^*(\omega_i)$.
It remains to observe that these elements are algebraically independent over $H$.
\end{proof}

Taking into account that $c_i$ restricted to $\bet{\left(\times_{i=1}^nO(1)\right)}$ is
equal to the $i$-th elementary symmetric function in $c_1$'s from components (and the fact that
$2=0$), we obtain also that $\eps^*(c_i)$ is either $u_i^2$, or $\tau\cdot u_i^2$, depending on
parity.
Theorem \ref{u} is proven.
\Qed


\begin{prop}
\label{BOcone}
In the category $\dmDVA{\bzar{O(n)}}$ of motives over $\bzar{O(n)}$ with $\zz/2$-coefficients
we have:
$$
M(\ezar{O(n)}\row\bzar{O(n)})=\otimes_{i=1}^n
\op{Cone}[-1]\left(\tate{}\stackrel{u_i}{\row}\tate{}([i/2])[i]\right),
$$
where $\tate{}=\tate{\bzar{O(n)}}$ is the Tate-motive.
\end{prop}

\begin{proof}
It follows by induction on $n$ from the exact triangles involving $\btzar{O(n-1)}$ and $\bzar{O(n)}$,
the fact that the functor $\ffi^*:\dmDVA{\bzar{O(n-1)}}\row\dmDVA{\btzar{O(n-1)}}$ maps
$M(\ezar{O(n-1)}\row\bzar{O(n-1)})$ to $M(\ezar{O(n-1)}\times\ezar{O(n)}\row\btzar{O(n-1)})$,
and the fact that
$u_i\in\hm^{i,[i/2]}(\bzar{O(i)},\zz/2)$ comes from $\bzar{O(n)}$.
\end{proof}

We call $u_i$ the {\it subtle Stiefel--Whitney classes}.
Clearly, under the topological realization functor these project to the
topological Stiefel--Whitney classes just as the usual Stiefel--Whitney classes $\omega_i$.
Notice that the motivic cohomology (with $\zz/2$-coefficients) for $\bzar{O(n)}$ look much
simpler than for $\bet{O(n)}$ (computed by N.Yagita in \cite[Theorem 8.1]{YaCF}).

The action of the Steenrod algebra is also as simple as in the topological case
(provided that $-1$ is a square in $\Field$).

\begin{prop}
\label{Stinrodu}
$$
Sq^k(u_m)=\sum_{j=0}^k\binom{-(m-k)}{j}u_{k-j}u_{m+j},\hspace{2mm}\text{for}\,\,k\leq m,
$$
and is zero for $k>m$.
\end{prop}

\begin{proof}
For $n=1$ we have $O(1)\cong\zz/2$, $Sq^1(u_1)=u_1^2$, and $Sq^k(u_1)=0$, for $k>1$.
The general case can be obtained
using Proposition \ref{i111n}.
By \cite[\S9]{VoOP}, we have a multiplicative operation
$R^{\bullet}=\sum_i(Sq^{2i}+Sq^{2i+1}\cdot\tau^{\frac{1}{2}})$.
Since $R^{\bullet}(u_1)=u_1+u_1^2\cdot\tau^{\frac{1}{2}}$, and $R^{\bullet}(\tau)=\tau$
(as $-1$ is a square in $k$), we get:
$$
R^{\bullet}(\delta^*(u_m))=\tau^{[m/2]}\cdot
\sigma_m(x_1(1+x_1\tau^{\frac{1}{2}}),\ldots,x_n(1+x_n\tau^{\frac{1}{2}})).
$$
It implies what we need (just as in topology), since $\binom{m-k+j-1}{j}=0$ mod $2$, if $[k/2]+[m/2]\neq [k-j/2]+[m+j/2]$.
\end{proof}

An important role in the computation of the map
$\hm^{*,*'}(\bzar{O(n)},\zz/2)\row\hm^{*,*'}(\hiq{X},\zz/2)$ will be played by the restrictions
induced by the embedding of groups $O(m)\times O(l)\subset O(n)$, where $n=m+l$.

\begin{prop}
\label{Olmn}
The map $\hm^{*,*'}(\bzar{O(n)},\zz/2)\row\hm^{*,*'}(\bzar{(O(m)\times O(l))},\zz/2)$ is given by:
$$
u_r\mapsto\sum_{i=0}^r u_i\otimes u_{r-i}\cdot\tau^{[\frac{r}{2}]-[\frac{i}{2}]-[\frac{r-i}{2}]}.
$$
\end{prop}

\begin{proof}
It follows immediately from Proposition \ref{i111n}.
\end{proof}

\subsection{Towards classification of quadratic forms}

Let us see how the above techniques can be used to classify torsors for an orthogonal group $G=O(n)$,
that is, quadratic forms.

Let $G=O(n)=O(q_n)$, where $q_n=\perp_{i=1}^n\la (-1)^{i-1}\ra$ is the standard split form of
dimension $n$. Everywhere below we assume that $(-1)$ is a square in $\Field$, so this form
coincides with the one considered in \ref{MotCohBOn}.
$G$-torsors are in $1$-to-$1$ correspondence with quadratic forms $q$, where
$X_q:=Iso(q\row q_n)$ is the variety of isomorphisms (of course, it has no rational point unless $q$ is
standard split).  It has the natural left action of $G=Iso(q_n\row q_n)$, as well as the right action
of $G_q=Iso(q\row q)$ (and it is a torsor under both).

To a quadratic form $q$ we can associate the variety of complete isotropic flags $F_q$.
If $\ddim(q)$ is even, then it has the property that, for any field extension $L/\Field$,
$$
F_q(L)\neq\emptyset\LRw q\cong\perp_j\hh\LRw
(X_q)_L\,\,\text{is trivial}\LRw X_q(L)\neq\emptyset.
$$
Thus, by (\ref{HsIsoOfEYToEXSSS}), in the case of even-dimensional $q$, $\hiq{X_q}=\hiq{F_q}$
(canonically).

If $\ddim(q)$ is odd, let $a=\op{det}_{\pm}(q)\in \Field^*/(\Field^*)^2$ be it's signed determinant. Then
$$
F_q(L)\neq\emptyset\,\,\text{and}\,\,X_{\la a\ra}(L)\neq\emptyset
\LRw q\cong(\perp_j\hh)\perp\la 1\ra\LRw
(X_q)_L\,\,\text{is trivial}\LRw X_q(L)\neq\emptyset.
$$
Thus, in this case, $\hiq{X_q}=\hiq{F_q}\times\hiq{X_{\la a\ra}}$, where $a=\op{det}_{\pm}(q)$.

We can see that the object $\hiq{X_q}$ (of $\hsF{\Field}$, or $\hF{\Field}$) itself carries the information of where $q$ is split,
but it does not remember $q$.

\begin{exa}
\label{adelim}
Let $q=\lva a\rva\cdot p$, where $\ddim(p)$ is odd. Then, for any extension $L/\Field$,
$$
q_L\,\,\text{is split}\,\,\LRw\lva a\rva_L\,\,\text{is split}.
$$
Thus, $\hiq{X_q}=\hiq{X_{\lva a\rva}}$.
\end{exa}

Still $\hiq{X_q}$ remembers various interesting things, for example, the $J$-invariant $J(q)$
(see \cite[Definition 5.11]{CGQG} for the definition). Recall, that the $\hs$-map
$\hiq{X_q}\stackrel{\ax{X_q}}{\row}\bzar{O(n)}$ does
remember $q$ itself.
\\

Now we can use $u_i$'s to reconstruct the motive of a torsor $X$ out of the motive of $\hiq{X}$.
We have the following diagram with cartesian squares in $\Spc$:
\begin{equation}
\label{hiXBOnDiagr}
\xymatrix @-0.7pc{
X\times\ezar{O(n)} \ar @{->}[r]  \ar @{->}[d]_(0.5){p}
& \egm{O(n)}\times\ezar{O(n)}  \ar @{->}[r] \ar @{->}[d] &
\ezar{O(n)} \ar @{->}[d]\\
\hiq{X} \ar @{->}[r] & O(n)\backslash(\egm{O(n)}\times\ezar{O(n)}) \ar @{->}[r] &
\bzar{O(n)},
}
\end{equation}
where the map $p$ is given by: $(x,g_0,g_1,\ldots,g_m)\mapsto(g_m^{-1}x,\ldots,g_0^{-1}x)$.


\begin{thm}
\label{XhiqX}
Let $X$ be an $O(n)$-torsor. Then in $\dmkDVA$,
$$
M(X)=\otimes_{i=1}^n\op{Cone}[-1]\left(M(\hiq{X})\stackrel{u_i(X)}{\lrow}M(\hiq{X})([i/2])[i]\right).
$$
\end{thm}

\begin{proof}
It follows from (\ref{hiXBOnDiagr}) that $M(X\times\ezar{O(n)}\row\hiq{X})$
is just the pull-back of $M(\ezar{O(n)}\row\bzar{O(n)})$
(whose description we know from Proposition \ref{BOcone})
under the restriction functor:
$$
\dmDVA{\bzar{O(n)}}\row\dmDVA{\hiq{X}}.
$$
This functor
respects tensor products. It remains to apply the forgetful functor
$$
\dmDVA{\hiq{X}}\row\dmkDVA
$$
which sends our motive to $M(X)$ and
also respects tensor products since the diagonal map
$M(\hiq{X})\row M(\hiq{X})\otimes M(\hiq{X})$ is an isomorphism
(see \cite[Lemma 6.8, Example 6.3]{VoMSS}).
\end{proof}

If $n=\ddim(q)$ is even, we have an action of $\gm(\Field)$ on $\hm^{*,*'}(\bzar{O(n)},\zz/2)$.
Let $q'=\lambda\cdot q$ be two proportional forms of the same (even) dimension $n$.
Then we have the canonical identification $\hiq{X_{q}}\cong\hiq{X_{q'}}$.

\begin{prop}
\label{lambdadejstvie}
There is a commutative diagram:
$$
\xymatrix @-0.7pc{
\hm^{*,*'}(\bzar{O(n)},\zz/2) \ar @{->}[r]^(0.5){\alpha_{X_q}^*}  \ar @{->}[d]_(0.5){\ffi_{\lambda}}
& \hm^{*,*'}(\hiq{X_{q}},\zz/2)  \ar @{=}[d] \\
\hm^{*,*'}(\bzar{O(n)},\zz/2)  \ar @{->}[r]_(0.5){\alpha_{X_{q'}}^*}
& \hm^{*,*'}(\hiq{X_{q'}},\zz/2),
}
$$
where $\ffi_{\lambda}$ is an automorphism over $H$ s. t.
$\ffi_{\lambda}(u_{2i+1})=u_{2i+1}$, $\ffi_{\lambda}(u_{2i})=u_{2i}+\{\lambda\}\cdot u_{2i-1}$.
\end{prop}

\begin{proof}
We start with the $1$-dimensional case. Consider quadratic forms
$\la 1\ra$ and $\la\lambda\ra$.
We have the following generalization of Theorem \ref{u}, which can be proven in exactly the same way.

\begin{prop}
\label{Yu}
Let ${\cal Y}$ be smooth simplicial scheme. Then
$$
\hm^{*,*'}({\cal Y}\times\bzar{O(n)},\zz/2)=\hm^{*,*'}({\cal Y},\zz/2)[u_1,\ldots,u_n],
$$
where $u_i$ are subtle Stiefel--Whitney classes.
\end{prop}

We have a torsor triple $(O(\la 1\ra),Y,O(\la\lambda\ra))$, where $Y=Iso(\la\lambda\ra\row\la1\ra)$.
$\hiq{Y}=\hiq{\lva\lambda\rva}$.
By Proposition \ref{Yu},
$\hm^{*,*'}(\hiq{Y}\times\bzar{O(\la 1\ra)},\zz/2)=\hm^{*,*'}(\hiq{\lva\lambda\rva},\zz/2)[u_1]$.
Our groups can be identified: $O(\la\lambda\ra)=O(\la 1\ra)=\zz/2$, and by Proposition \ref{GYH}
we have an identification:
$$
\hiq{Y}\times\bzar{O(\la\lambda\ra)}\stackrel{\theta^{Zar}_Y}{\cong}\bzar{O(\la 1\ra)}\times\hiq{Y}.
$$

\begin{lem}
\label{hiqa1}
$\hm^{*,*'}(\hiq{\lva\lambda\rva},\zz/2)=H[\gamma]/(\tau\cdot\gamma=\{\lambda\};\,
\gamma\cdot Ann_{\{\lambda\}}=0)$,
where $\gamma\in\hm^{1,0}(\hiq{\lva\lambda\rva},\zz/2)$ and $Ann_{\{\lambda\}}$ is the annulator of
$\{\lambda\}$ in $K^M_*(\Field)/2$.
\end{lem}

\begin{proof}
In $\dmDVA{\Field}$ we have an exact triangle (cf. \cite[Theorem 4.4]{VoMil}):
$$
\xymatrix @-2.1pc{
& M(\op{Spec}(\Field\sqrt{\lambda})) \ar @{->}[rddd] &  \\
& & \\
& \star & \\
M(\hiq{\lva\lambda\rva}) \ar @{->}[ruuu] &  & M(\hiq{\lva\lambda\rva})
\ar @{->}[ll]_(0.5){\gamma}^(0.5){[1]}.
}
$$
The map $\gamma$ here is given by the same-named element
$\gamma\in\hm^{1,0}(\hiq{\lva\lambda\rva},\zz/2)$, and since $\op{Spec}(\Field\sqrt{\lambda})$
is a zero-dimensional pure motive, multiplication by $\gamma$ is an isomorphism on all
diagonals starting from the $1$-st one, and a surjection on the $0$-th diagonal.
On the other hand, it follows from \cite{VoMil} that multiplication by $\tau$
identifies the $1$-st diagonal with the $\kker(K^M_*(\Field)/2\row K^M_*(\Field\sqrt{\lambda})/2)$,
which is a principal ideal in $K^M_*(\Field)/2$ generated by $\{\lambda\}$.
Thus, for each $i\geq 1$, the $i$-th diagonal is a cyclic module over $K^M_*(\Field)/2$
generated by $\gamma^i$ and isomorphic to $\{\lambda\}\cdot K^*(\Field)/2$.
Clearly, $\tau\cdot\gamma=\{\lambda\}$. The rest of the description follows.
\phantom{a}\hspace{5mm}
\end{proof}

Let $\hiq{Y}\stackrel{\ax{Y}}{\row}\bzar{O(\la 1\ra)}$ be the fiber over $[Y]$.


\begin{lem}
\label{gammau1}
We have: $\alpha_Y^*(u_1)=\gamma$. In particular, the map
$$
\ax{Y}^*:\hm^{*,*'}(\bzar{O(\la 1\ra)},\zz/2)\row\hm^{*,*'}(\hiq{\lva\lambda\rva},\zz/2)
$$
is surjective.
\end{lem}

\begin{proof}
We know that $\ax{Y}^*(\omega_1)=\{\lambda\}$. Hence, $\ax{Y}^*(u_1)=\gamma$.
\end{proof}

It follows from Proposition \ref{GYH} that
$(\theta^{Zar}_Y)^*(u_1)=u_1+\gamma$.

Let now $q_n=\perp_{i=1}^n\la 1\ra$, and $q'_n=\lambda\cdot q_n$.
Consider the torsor triple $(O(q_n),X,O(q'_n))$, where $X=Iso(q'_n\row q_n)$.
Since $q_n$ is even-dimensional split, the torsor $X$ is trivial. Hence,
by Proposition \ref{GYH}, we have an identification
$
\xymatrix @-0.7pc{
\bzar{O(q'_n)} \ar @{->}[r]^(0.5){\theta^{Zar}_X}_(0.5){=} & \bzar{O(q_n)}.
}
$
It can be extended to a commutative diagram:
$$
\xymatrix @-0.7pc{
\hiq{\lva\lambda\rva}\times\bzar{O(q_n')} \ar @{->}[r]^(0.5){\theta} &
\bzar{O(q_n)}\times\hiq{\lva\lambda\rva} \\
\hiq{\lva\lambda\rva}\times\times_{j=1}^n\bzar{O(\la\lambda\ra)}
\ar @{->}[r]_(0.5){\wt{\theta}} \ar @{->}[u] &
\times_{j=1}^n\bzar{O(\la 1\ra)}\times\hiq{\lva\lambda\rva} \ar @{->}[u],
}
$$
Computing the induced maps on $u_r$ we get:
$$
\xymatrix @-0.7pc{
\sum_{j=0}^r\binom{n-j}{r-j}u_j\cdot\gamma^{r-j}\cdot\tau^{[r/2]-[j/2]} \ar @{|->}[d] &
u_r \ar @{|->}[d] \\
\sigma_r(x_1+\gamma,\ldots,x_n+\gamma)\cdot\tau^{[r/2]} &
\sigma_r(x_1,\ldots,x_n)\cdot\tau^{[r/2]} \ar @{|->}[l].
}
$$
Therefore,
$$
\theta^*(u_r)=\sum_{j=0}^r\binom{n-j}{r-j}u_j\cdot\gamma^{r-j}\cdot\tau^{[r/2]-[j/2]}.
$$
Notice that $\gamma^2\cdot\tau=\gamma\cdot\{-1\}$, and since $-1$ is a square in $\Field$,
we obtain: $\theta^*(u_{2m+1})=u_{2m+1}$ and
$\theta^*(u_{2m})=u_{2m}+u_{2m-1}\cdot\{\lambda\}$.
Since the map:
$$
\hm^{*,*'}(\bzar{O(q'_n)},\zz/2)\row
\hm^{*,*'}(\hiq{\lva\lambda\rva}\times\bzar{O(q'_n)},\zz/2)
$$
is injective
by Proposition \ref{Yu}, the same formulas work for the map $(\theta^{Zar}_X)^*$.
If now $q$ is an arbitrary $n$-dimensional form, and $q'=\lambda\cdot q$, then
the torsor $Z'=Iso(q'\row q'_n)$ can be canonically identified with the torsor
$Z=Iso(q\row q_n)$ so that we have a commutative diagram:
$$
\xymatrix @-0.7pc{
\hiq{Z'} \ar @{=}[r] \ar @{->}[d]_(0.5){\ax{Z'}^*}& \hiq{Z} \ar @{->}[d]^(0.5){\ax{Z}^*}\\
\bzar{O(q_n')} \ar @{=}[r] &
\bzar{O(q_n)},
}
$$
where we use the standard identification $O(q'_n)=O(q_n)$ (not $\theta^{Zar}_X$).
It remains to observe that the composition $\hiq{X_{q}}\stackrel{\ax{Z'}^*}{\row}
\bzar{O(q'_n)}\stackrel{\theta^{Zar}_X}{\lrow}\bzar{O(q_n)}$ is equal to the composition:
$
\xymatrix @-0.7pc{
\hiq{X_{q}} \ar @{=}[r] & \hiq{X_{q'}} \ar @{->}[r]^(0.4){\ax{X_{q'}}^*} & \bzar{O(q_n)}.
}
$
Proposition \ref{lambdadejstvie} is proven.
\end{proof}

We can use the {\it subtle Stiefel--Whitney classes} to reconstruct the motive
of the highest quadratic Grassmanian corresponding to $X=X_q$.
Let $d=[n/2]$ and $P_d=P_d^n$ be the stabilizator in $O(n)$ of the totally-isotropic
subspace of maximal dimension (equal to $d$). Then $G_d(X)=P_d\backslash X$ is the
variety of projective subspaces of maximal dimension on the projective quadric $Q$
(defined by the form $q$).
We have a natural embedding $GL_d\subset P_d$, for even $n$, and
$GL_d\times\zz/2\subset P_d$, for odd $n$, coming from the presentation of
$V_{{q_n}}$ as $V\oplus V^*$ and $V\oplus V^*\oplus V_{\la 1\ra}$, respectively.

The cohomology of $\bzar{(GL_d)}$ can be computed in the same way as for $\bzar{O(n)}$.


\begin{prop}
\label{GLd}
$$
\hm^{*,*'}(\bzar{(GL_d)},\zz)=\hm^{*,*'}(\op{Spec}(\Field),\zz)[c_1,c_2,\ldots,c_d],
$$
where $c_i\in\hm^{2i,i}(\bzar{(GL_d)},\zz)$ coincides with the pull-back of the
$i$-th Chern class from $\bet{(GL_d)}$. Also, in $\dm{\bzar{(GL_d)}}$,
$$
M(\ezar{(GL_d)}\row\bzar{(GL_d)})=\otimes_{j=1}^d
\op{Cone}[-1]\left(\tate{}\stackrel{c_j}{\lrow}\tate{}(j)[2j]\right),
$$
where $\tate{}=\tate{\bzar{(GL_d)}}$ is the Tate-motive.
\end{prop}

\begin{proof}
From the tower of subgroups
$$
\{e\}\row GL_1\row GL_2\row\ldots\row GL_{d-1}\row GL_d,
$$
we get a tower of fibrations:
$$
\ezar{(GL_d)}\stackrel{g_1}{\lrow}\bzart{(GL_1)}\stackrel{g_2}{\lrow}\bzart{(GL_2)}\stackrel{}{\lrow}
\ldots\stackrel{}{\lrow}\bzart{(GL_{d-1})}\stackrel{g_d}{\lrow}\bzar{(GL_d)},
$$
(where $\bzart{(GL_i)}$ is $\hs$-isomorphic to $\bzar{(GL_i)}$)
which is trivial over simplicial components.
Since $GL_{d-1}\backslash GL_d\cong\aaa^d\backslash 0$, whose motive in $\dm{\Field}$ is $\zz\oplus\zz(d)[2d-1]$,
we get that in $\dm{\bzar{(GL_d)}}$ there is a distinguished triangle
$$
\xymatrix @-1.7pc{
& M(\bzart{(GL_{d-1})}) \ar @{->}[rddd]^(0.5){g_d} &  \\
& & \\
& \star & \\
M(\bzar{(GL_d)})(d)[2d] \ar @{->}[ruuu]^(0.5){[1]} &  & M(\bzar{(GL_d)})
\ar @{->}[ll].
}
$$
And again, by induction on $d$ and degree considerations, we get that $g_d^*$ is surjective on
$\hm^{*,*'}(-,\zz)$ with the kernel generated by $c_d\in\hm^{2i,i}(\bzar{(GL_d)},\zz)$, and
$c_1,\ldots,c_{d-1}$ are lifted uniquely to $\hm^{*.*'}(\bzar{(GL_d)},\zz)$.
Simultaneously, we get that in $\dm{\bzar{(GL_d)}}$,
$$
M(\bzart{(GL_{d-1})}\row \bzar{(GL_{d})})=
\op{Cone}[-1]\left(\tate{}\stackrel{c_d}{\lrow}\tate{}(d)[2d]\right).
$$
Recalling that $c_i$ comes from $\bzar{(GL_d)}$, we get the description of
$M(\ezar{(GL_d)}\row\bzar{(GL_d)})$.
\end{proof}

\begin{rem}
\label{GLdetzar}
In particular, the map
$\eps^*:\hm^{*.*'}(\bet{(GL_d)},\zz)\row\hm^{*.*'}(\bzar{(GL_d)},\zz)$
is an isomorphism. But, actually, the very map
$\eps:\bzar{(GL_d)}\row\bet{(GL_d)}$ is an isomorphism in $\hsF{\Field}$
by {\rm \cite[Lemma 4.1.18]{MV}}.
\end{rem}


\begin{prop}
\label{PdGLd}
In $\hk$, we have an identification: $\bzar{(P^n_d)}=\bzar{(GL_d)}$, for even $n$, and
$\bzar{(P^n_d)}=\bzar{(GL_d\times\zz/2)}$, for odd $n$.
In particular,
\begin{equation*}
\hm^{*,*'}(\bzar{(P^n_d)},\zz/2)=
\begin{cases}
H[c_1,\ldots,c_d],\,\,\,
\text{if}\,\,\,n\,\,\,\text{is even};\\
H[u_1,c_1,\ldots,c_d],
\,\,\,\text{if}\,\,\,n\,\,\,\text{is odd}.
\end{cases}
\end{equation*}
\end{prop}

\begin{proof}
For $n$ even, we have a decomposition: $P^n_d=GL_d\cdot U$, where
$U$ consists of transformations: $(v,v^*)\mapsto (v+f(v^*),v^*)$, where
$f:V^*\row V$ is a linear map with the property: $<f(v^*),v^*>=0$.

For $n$ odd, we have a decomposition:
$P^n_d=(GL_d\times\zz/2)\cdot U'$, where $U'$ consists of transformations:
$(v,v^*,\alpha)\mapsto (v+f(v^*)+\alpha\cdot w,v^*,\alpha-\frac{1}{2}<w,v^*>)$,
where $w\in V$ is an arbitrary vector, and
$f:V^*\row V$ is a linear map with the property:\\
$<f(v^*),v^*>=-\frac{1}{4}(<w,v^*>)^2$.

Since $U$ and $U'$ are isomorphic to affine spaces, we have an identification in $\hk$:
\begin{equation*}
\bzar{(P^n_d)}=
\begin{cases}
\bzar{(GL_d)},\,\,\,
\text{if}\,\,\,n\,\,\,\text{is even};\\
\bzar{(GL_d\times\zz/2)},\,\,\,
\text{if}\,\,\,n\,\,\,\text{is odd}.\\
\end{cases}
\end{equation*}
\end{proof}

We have the following natural diagram in $\Spc$:
$$
\bzar{P^n_d}\stackrel{\ffi}{\llow}\btzar{P^n_d}\stackrel{\psi}{\lrow}
\bttzar{P^n_d},
$$
where $\btzar{P^n_d}=P^n_d\backslash(\ezar{P^n_d}\times\ezar{O(n)})$, and
$\bttzar{P^n_d}=P^n_d\backslash\ezar{O(n)}$.
Since the projection $O(n)\row P^n_d\backslash O(n)$ is split over each point of the base,
the maps $\ffi$ and $\psi$ are isomorphisms in $\hsF{\Field}$. In particular,
$M(\btzar{P^n_d}\stackrel{\psi}{\row}\bttzar{P^n_d})=T\in\dmDVA{\bttzar{P^n_d}})$.

\begin{prop}
\label{Pdu}
Under the natural projection $\bttzar{(P^n_d)}\stackrel{f}{\row}\bzar{O(n)}$, we have:
$f^*(u_{2i})=c_i$, and $f^*(u_{2i+1})=0$, for even $n$, and $f^*(u_{2i+1})=c_i\cdot u_1$, for odd $n$.
In particular,
\begin{equation*}
\hm^{*,*'}(\bzar{(P^n_d)},\zz/2)=
\begin{cases}
\hm^{*,*'}(\bzar{O(n)},\zz/2)/(u_{2i+1},\,\,0\leq i<n/2),\,\,\,
\text{if}\,\,\,n\,\,\,\text{is even};\\
\hm^{*,*'}(\bzar{O(n)},\zz/2)/(u_{2i+1}-u_{2i}\cdot u_1,\,\,0<i<n/2),
\,\,\,\text{if}\,\,\,n\,\,\,\text{is odd}.
\end{cases}
\end{equation*}
\end{prop}

\begin{proof}
Consider the diagram of group embeddings:
$$
\xymatrix @-0.7pc{
P^{n-2}_{d-1} \ar @{->}[r] \ar @{->}[d] & P^n_d \ar @{->}[d] \\
O(n-2) \ar @{->}[r] & O(n).
}
$$
If $n$ is even, then the composition
$$
GL_{d-1}\backslash GL_d\row P^{n-2}_{d-1}\backslash P^n_d
\row O(n-2)\backslash O(n)\row O(n-1)\backslash O(n)
$$
is an isomorphism in $\dm{\Field}$.
If $n$ is odd, then the composition
$$
GL_{d-1}\backslash GL_d\row P^{n-2}_{d-1}\backslash P^n_d\row O(n-2)\backslash O(n)
$$
factors through the map $GL_{d-1}\backslash GL_d\row O(n-2)\backslash O(n-1)$,
which is an isomorphism in $\dm{\Field}$.
This implies that $c_i=f^*(u_{2i})$. The fact that $f^*$ of odd subtle Stiefel--Whitney classes
is zero, for even $n$, follows from the fact that there are no elements of such grading in
$\hm^{*,*'}(\bzar{(GL_d)},\zz/2)$. And for odd $n$, we observe that the map
$GL_d\times\zz/2\row O(n)$ factors through $O(n-1)\times O(1)\row O(n)$. It remains to apply
Proposition \ref{Olmn}.
\end{proof}

Note that the fibration $\bttzar{(P^n_d)}\row\bzar{O(n)}$ is trivial over the graded
components with fibers -
the split Grassmannian $P^n_d\backslash O(n)$. In particular, the graded components
$$
r_i^*M(\bttzar{(P^n_d)}\row\bzar{O(n)})\in\dmDVA{(\bzar{O(n)})_i}
$$
belong to the thick
subcategory $DT((\bzar{O(n)})_i)$
generated by Tate-motives.
We have the following general fact:

\begin{prop}
\label{gradDT}
Let $F$ be a field of characteristic $p$, $\cy$ be smooth simplicial scheme over
smooth scheme $S$ such that $\pi=\pi_1(CC(\cy))$ is a $p$-group. Let $N\in\dmR{\cy}{F}$ be such motive that its graded components $N_i\in\dmR{\cy_i}{F}$
belong to $DT(\cy_i)$, and for each $\theta:[i]\row [j] \in Mor(\Delta)$
the natural map $f_{\theta}:(\cy_{\theta}^*)(N_j)\row N_i$ is an isomorphism. Then $N\in DT(\cy)$.
\end{prop}

\begin{proof}
By \cite[Lemma 5.9]{VoMSS} we have a slice filtration on each of $N_i$ with only finitely many nontrivial
graded pieces $s_m(N_i)\in DT_m(\cy_i)$. Moreover, for each $\theta:[i]\row [j] \in Mor(\Delta)$, we
have the natural map $f_{\theta}:(\cy_{\theta}^*)(N_j)\row N_i$ which uniquely extends to the filtration
according to \cite[Lemma 5.11]{VoMSS}. It follows from \cite[Remarks 5.19, 5.21]{VoMSS} and our
conditions on $F$ and $\cy$ (which guarantee that any $F[\pi]$-module is an extension of trivial ones)
that $DT'_0(\cy)=DLC(\cy)$ (loc. cit.). Let $l=\op{min}(m|\,s_m(N_i)\neq 0)$ (for some $=$ for any $i$).
Then $s_l(N_i)$ considered as an element of $DT'_0(\cy_i)$ (see \cite[Proposition 5.20]{VoMSS}) has natural
filtration coming from the $t$-structure with the heart $LC(\cy_i)=F-mod$ (\cite[Remark 5.21]{VoMSS}).
This filtration is respected by the maps $f_{\theta}$, which provide the action of $\pi$ on the
$t$-graded pieces $(s_l(N_i))_k$. If $r=\op{min}(k|\,(s_l(N_i))_k\neq 0)$, then representing $(s_l(N_i))_r$
as an extension of trivial $\pi$-modules, we get the natural map
$(s_l(N_i))_r\row \oplus T(l)[r]$ (the coinvariants of the $\pi$-action)
which is respected by the maps $f_{\theta}$ and so gives the map
$Lc_{\#}N\row \oplus T(l)[r]$ in $\dmR{S}{F}$ (here we are using arguments from the proof of Proposition
\ref{gradTate}). By conjugation we get the map
$\ffi:N\row \oplus T(l)[r]$. Clearly the motive $N':=Cone[-1](\ffi)$ still satisfies the conditions of our
Proposition. Repeating this process we will eventually kill the $(l)[r]$-component of $N$, and the
induction on $l$ and $r$ finishes the proof.
\end{proof}

Now we can compute the motive of $\bttzar{(P^n_d)}$ in $\dmDVA{\bzar{O(n)}}$.

\begin{prop}
\label{PdOn}
In $\dmDVA{\bzar{O(n)}}$, we have the natural identification:
$$
M(\bttzar{(P^n_d)}\row\bzar{O(n)})=\otimes_{\delta\leq j<n/2}
\op{Cone}[-1]\left(\tate{}\stackrel{v_{2j+1}}{\lrow}\tate{}(j)[2j+1]\right),
$$
where $v_{2j+1}=u_{2j+1}$, for even $n$, $=u_{2j+1}-u_{2j}\cdot u_1$, for odd $n$, and
$\delta=n-2d$.
\end{prop}

\begin{proof}
Since $f^*(v_{2j+1})=0$, and $v_{2j+1}$'s form a regular sequence, it follows that the map
$$
M(\bttzar{(P^n_d)}\stackrel{f}{\row}\bzar{O(n)})\row\tate{}
$$
factors through
$\otimes_{\delta\leq j<n/2}
\op{Cone}[-1]\left(\tate{}\stackrel{v_{2j+1}}{\lrow}\tate{}(j)[2j+1]\right)$,
and the respective map induces an isomorphism on $\hm^{*,*'}(-,\zz/2)$.
Observing that $CC(\bzar{O(n)})=K(\zz/2,1)$, from Proposition \ref{gradDT} we obtain that $M(\bttzar{(P^n_d)}\row\bzar{O(n)})$ belongs to
$DT(\bzar{O(n)})$. So, we have a morphism between two objects of $DT$ which gives an isomorphism
on cohomology. It must be an isomorphism by \cite[Lemma 5.2]{VoMSS}. Thus,
$$
M(\bttzar{(P^n_d)}\row\bzar{O(n)})=\otimes_{\delta\leq j<n/2}
\op{Cone}[-1]\left(\tate{}\stackrel{v_{2j+1}}{\lrow}\tate{}(j)[2j+1]\right),
$$
in $\dmDVA{\bzar{O(n)}}$.
\end{proof}

Now we can compute the motive of $G_d$.

\begin{thm}
\label{GdhiqEV}
Let $q$ be a quadratic form of even dimension $n=2d$, and $G_d(q)$ be it's highest quadratic Grassmannian.
Then, in $\dmkDVA$,
$$
M(G_d(q))=\otimes_{0\leq j\leq d-1}
\op{Cone}[-1]\left(M(\hiq{G_d(q)})\stackrel{u_{2j+1}(q)}{\lrow}M(\hiq{G_d(q)})(j)[2j+1]\right).
$$
\end{thm}

\begin{proof} Let $X=X_q$ be the respective $O(n)$-torsor.
We have a diagram with cartesian squares in $\Spc$:
$$
\begin{CD}
\ezar{O(n)} & @>>> & \bttzar{(P^n_d)} & @>>> & \bzar{O(n)}\\
@AAA & & @AAA & & @AAA \\
X\times\ezar{O(n)} & @>>> & P^n_d\backslash (X\times\ezar{O(n)}) & @>>> & \hiq{X},
\end{CD}
$$
coming from the $P^n_d$ and $O(n)$ actions.
Denote $\tilde{G}_d:=P^n_d\backslash (X\times\ezar{O(n)})$.
We obtain that $M(\tilde{G}_d\row\hiq{X})$ in $\dmDVA{\hiq{X}}$ is simply the image of the
$M(\bttzar{(P^n_d)}\row\bzar{O(n)})$ under the natural functor
$\dmDVA{\bzar{O(n)}}\row\dmDVA{\hiq{X}}$. Since this functor respects the tensor product,
Proposition \ref{PdOn} implies that
$$
M(\tilde{G}_d\row\hiq{X})=\otimes_{0\leq j\leq d-1}
\op{Cone}[-1]\left(\tate{\hiq{X_q}}\stackrel{u_{2j+1}(q)}{\lrow}\tate{\hiq{X_q}}(j)[2j+1]\right).
$$
It remains to apply the forgetful functor $\dmDVA{\hiq{X}}\row\dmkDVA$, which also respects
the tensor product, since the diagonal map $M(\hiq{X})\stackrel{\Delta}{\row} M(\hiq{X})\otimes M(\hiq{X})$
is an isomorphism (by \cite[Lemma 6.8, Example 6.3]{VoMSS}). The result will be $M(\tilde{G}_d)$ which coincides with $M(G_d)$ since
the projection $X\row G_d$ is split over every point
(notice, that this would not work for other Grassmannians, or for odd $n$).
\end{proof}

There is an odd-dimensional variant as well.
Let $q$ be a form of odd dimension $n=2d+1$, and $p=q\perp\la a\ra$, where $a=\ddet_{\pm}(q)$
be an $(n+1)$-dimensional form from $I^2$, containing it.
Then $\hiq{X_q}=\hiq{X_p}\times\hiq{\{a\}}$, and it follows from Proposition \ref{Olmn} that,
for $\hiq{X_q}\stackrel{\nu}{\row}\hiq{X_p}$, $\nu^*(u_{2j+1}(p))=u_{2j+1}(q)+u_{2j}(q)\cdot u_1(q)$.
Taking into account that $G_{d+1}(p)=G_d(q)\coprod G_d(q)$, and $\hiq{X_p}=\hiq{G_d(q)}$ we get:

\begin{prop}
\label{GdhiqOD}
Let $q$ be a form of odd dimension $n=2d+1$, and $p=q\perp\la\ddet_{\pm}(q)\ra$. Then the motive
of the highest Grassmannian of $q$ can be presented as:
$$
M(G_d(q))=\otimes_{1\leq j\leq d}
\op{Cone}[-1]\left(M(\hiq{G_d(q)})\stackrel{u_{2j+1}(p)}{\lrow}M(\hiq{G_d(q)})(j)[2j+1]\right).
$$
\end{prop}

\begin{exa}
Let $q=\la a,b,-ab,-c,-d,cd\ra$ be an Albert form. Then $d=3$, and $G_3(q)=S\coprod S$, where
$S=SB(\{a,b\}+\{c,d\})$ is the Severi-Brauer variety corresponding to the element
$\{a,b\}+\{c,d\}\in K^M_2(k)/2$. It follows from the above that
$$
M(S)=\op{Cone}[-1]\left(M(\hiq{S})\stackrel{u_{3}(q)}{\lrow}M(\hiq{S})(1)[3]\right)\otimes
\op{Cone}[-1]\left(M(\hiq{S})\stackrel{u_{5}(q)}{\lrow}M(\hiq{S})(2)[5]\right).
$$
Even in this simple case, the decomposition into tensor product of binary motives was unknown (though, expected ... for 18 years).
\end{exa}

\begin{rem}
Another case where the presentation of the motive of a variety as an extension of motives
of Chech simplicial schemes is known is the case of a quadric. The canonical decomposition there
was obtained in {\rm \cite[Theorems 3.1, 3.7]{IMQ}}. Though, in the case of the highest quadratic
Grassmannian above we get nice poly-binary structure with the precise description of
connections involved and all
elementary pieces of the same kind (as opposed to the case of a quadric), which is related to
the fact that $G_d(q)$ is generically split.
\end{rem}

We have the following "flexible" versions of Proposition \ref{BOcone}, Proposition \ref{PdOn}, and
Theorem \ref{XhiqX}, Theorem \ref{GdhiqEV} (Proposition \ref{GdhiqOD}), respectively.

\begin{prop}
\label{flexX}
Let $\wt{u_i}=u_i+\text{decomposable terms}\in\hm^{i,[i/2]}(\bzar{O(n)},\zz/2)$, for $i=1,\ldots,n$ be some elements. Then
in $\dmDVA{\bzar{O(n)}}$ and in $\dmDVA{k}$, respectively:\\
\vspace{-5mm}
\begin{equation*}
\begin{split}
&(1)\hspace{2cm} M(\ezar{O(n)}\row\bzar{O(n)})=\otimes_{i=1}^n
\op{Cone}[-1]\left(\tate{}\stackrel{\wt{u_i}}{\row}\tate{}([i/2])[i]\right),\\
&(2)\hspace{2cm}
M(X)=\otimes_{i=1}^n\op{Cone}[-1]\left(M(\hiq{X})\stackrel{\wt{u_i}(X)}{\lrow}M(\hiq{X})([i/2])[i]\right).
\hspace{3cm}\phantom{a}
\end{split}
\end{equation*}
\end{prop}

\begin{proof}
Since the sequence $\wt{u_i},\,i=1,\ldots,n$ is regular, the natural map
$M(\ezar{O(n)}\row\bzar{O(n)})\row\tate{}$ can be factored through
$\otimes_{i=1}^n\op{Cone}[-1]\left(\tate{}\stackrel{\wt{u_i}}{\row}\tate{}([i/2])[i]\right)$
inducing an isomorphism on $\hm^{*,*'}$. By Proposition \ref{BOcone},
$M(\ezar{O(n)}\row\bzar{O(n)})$ belongs to the thick subcategory $DT(\bzar{O(n)})$ generated
by Tate-motives. By \cite[Lemma 5.2]{VoMSS}, our map is an isomorphism. This settles 1). Then 2) follows
as in the proof of Theorem \ref{XhiqX}.
\end{proof}

The case of Grassmannians can be done in exactly the same way
(we formulate the even dimensional case only, the other one is analogous):

\begin{prop}
\label{flexGd}
Let $n=2d$ be even,
and
$$
\wt{u_{2j+1}}=u_{2j+1}+\text{decomposable terms}\in\hm^{2j+1,j}(\bzar{O(n)},\zz/2), \quad (j=0,\ldots,d-1).
$$
Then
in $\dmDVA{\bzar{O(n)}}$ and in $\dmDVA{k}$, respectively:
\begin{equation*}
\begin{split}
&(1)\hspace{1cm} M(\bttzar{(P^n_d)}\row\bzar{O(n)})=\otimes_{0\leq j\leq d-1}
\op{Cone}[-1]\left(\tate{}\stackrel{\wt{u_{2j+1}}}{\lrow}\tate{}(j)[2j+1]\right),\\
&(2)\hspace{1cm} M(G_d(q))=\otimes_{0\leq j\leq d-1}
\op{Cone}[-1]\left(M(\hiq{G_d(q)})\stackrel{\wt{u_{2j+1}}(q)}{\lrow}M(\hiq{G_d(q)})(j)[2j+1]\right).
\end{split}
\end{equation*}
\end{prop}

Theorem \ref{GdhiqEV} and Proposition \ref{GdhiqOD} permit us to connect the subtle
Stiefel--Whitney classes with the $J$-invariant of $q$ (see \cite{CGQG}).

\begin{prop}
\label{minJU}
Let $q$ be a quadratic form of dimension $n$, and $p=q$, if $n$ is even, and
$p=q\perp\la\ddet_{\pm}(q)\ra$, if $n$ is odd. Then:
$$
\op{min}\{j|\,j\notin J(q)\}=\op{min}\{j|\,u_{2j+1}(p)\neq 0\}.
$$
\end{prop}

\begin{proof}
By the Main Theorem of \cite{CGQG}, $\op{min}\{j|\,j\notin J(q)\}$ is equal to the minimal codimension
of non-rational class in $\op{CH}^*(G_d(q)|_{\kbar})/2$. By Theorem \ref{GdhiqEV} (Proposition \ref{GdhiqOD}, respectively), this number is also equal to
$\op{min}\{j|\,u_{2j+1}(p)\neq 0\}$.
\end{proof}

But there are much more precise statements.
In $\dmDVA{\Field}$ consider objects:
$$
C_{2l+1}:=\op{Cone}[-1]\left(\hiq{X_q}\stackrel{u_{2l+1}(q)}{\lrow}\hiq{X_q}(l)[2l+1]\right).
$$
Let $N_{j-1}:=\otimes_{0\leq l<j}C_{2l+1}$, and $f_j$ be the composition
$N_{j-1}\row\hiq{X_q}\stackrel{u_{2j+1}(q)}{\lrow}\hiq{X_q}(j)[2j+1]$.

\begin{prop}
\label{zf}
Let $q$ be an $n=2d$-dimensional quadratic form, and $0\leq j< d$.
Then the following conditions are equivalent:
\begin{itemize}
\item[$(1)$ ] $j\in J(q)$;
\item[$(2)$ ] The map $f_j:N_{j-1}\row\hiq{X_q}(j)[2j+1]$ is zero.
\end{itemize}
\end{prop}

\begin{proof}
In the category $\dmDVA{\bzar{O(n)}}$ of motives over $\bzar{O(n)}$ consider objects:
$\hat{C}_{2l+1}:=\op{Cone}[-1]\left(\tate{}\stackrel{u_{2l+1}}{\lrow}\tate{}(l)[2l+1]\right)$.
Then the natural map $M(\bttzar{(P^n_d)}\row\bzar{O(n)})\row T$ can be lifted to the map
$$
\rho_l:M(\bttzar{(P^n_d)}\row\bzar{O(n)})\row\hat{C}_{2l+1}.
$$
This lifting is defined up to the choice
of element of
$$
\hm^{2l,l}(\bzar{(P^n_d)},\zz/2)=H[u_2,u_4,\ldots,u_{2d}]_{(l)[2l]}=
\zz/2[u_2,u_4,\ldots,u_{2d}]_{(l)[2l]}.
$$
The composition:
$$
M(\bttzar{(P^n_d)}\row\bzar{O(n)})\stackrel{\Delta_d}{\lrow}
M(\bttzar{(P^n_d)}\row\bzar{O(n)})^{\otimes d}\stackrel{\otimes\rho_l}{\lrow}
\otimes_{0\leq l<d}\hat{C}_{2l+1}
$$
is a choice of isomorphism of Proposition \ref{PdOn}.
Let $(\bzar{O(n)})_0=\op{Spec}(\Field)=\bullet$ be the $0$-th graded
component of the simplicial scheme $\bzar{O(n)}$. Then $(\hat{C}_{2l+1})_0$ uniquely splits as
$T\oplus T(l)[2l]$, since the Subtle Stiefel--Whitney classes are trivial when restricted to
$(\bzar{O(n)})_0$, and there are no maps between the Tate-motives involved
(both by degree consideration, for example). At the same time,
$(M(\bttzar{(P^n_d)}\row\bzar{O(n)}))_0$ is the motive $M(G_d(q_n))$ of the split (highest) Grassmannian
in $\dmDVA{\Field}$. Thus, we get the canonical map
$\phi_l: M(G_d(q_n))\row T(l)[2l]$ giving the class in $\op{CH}^l(G_d(q_n))/2$.

The Chow ring of the split quadratic Grassmannian is generated by the special
"elementary classes" $z_l$ - see \cite[Proposition 2.4]{CGQG}, or \cite[Section 2]{u}.

\begin{lem}
\label{z-l}
We have: $\phi_l=z_l$.
\end{lem}

\begin{proof}
We have a commutative diagram in $\Spc$:
$$
\xymatrix @-0.7pc{
\bttzar{P^{2l}_l} \ar @{->}[r] \ar @{->}[d] & \bzar{O(2l)} \ar @{->}[d]^(0.5){f} &
\bullet \ar @{->}[l]_(0.3){g} \ar @{->}[ld]^(0.5){h}\\
\bttzar{P^{n}_d} \ar @{->}[r] & \bzar{O(n)},
}
$$
which induces the map:
$M(\bttzar{P^{2l}_l}\row\bzar{O(2l)})\lrow f^*M(\bttzar{P^{n}_d}\row\bzar{O(n)})$.
If we apply $g^*$ to it we will get the natural map $M(G_l(q_{2l}))\stackrel{p}{\lrow} M(G_d(q_n))$.
We have a natural identification $f^*(\hat{C}_{2l+1})=\hat{C}_{2l+1}$. Thus, our lifting
$f^*M(\bttzar{P^{n}_d}\row\bzar{O(n)})\row f^*(\hat{C}_{2l+1})$ will restrict to
the lifting $M(\bttzar{P^{2l}_l}\row\bzar{O(2l)})\row\hat{C}_{2l+1}$. But $\hat{C}_{2l+1}$ in
$\dmDVA{\bzar{O(2l)}}$ is split since $f^*u_{2l+1}=0$. Thus, the projection to $T(l)[2l]$ is
defined already on the level of $M(\bttzar{P^{2l}_l}\row\bzar{O(2l)})$ and so is a polynomial
in $u_{2i},\,1\leq i\leq l$ with $\zz/2$-coefficients. But all these classes vanish under $g^*$.
Hence, $p^*(\phi_l)=0$, and so $\phi_l=z_l\in\op{CH}^l(G_d(q_n))/2$ by
\cite[Proposition 2.4(3)]{CGQG}.
\end{proof}

Apply the motivic restriction
$\alpha_X^*: \dmDVA{\bzar{O(n)}}\to\dmDVA{\hiq{X_q}}$,
corresponding to the map
$\alpha_X: \hiq{X_q}\to\bzar{O(n)}$.
Notice that this map respects tensor products.
Denote the image of $M(\bttzar{(P^n_d)}\row\bzar{O(n)})$ as $M$, and
the image of $\hat{C}_{2l+1}$ as $C_{2l+1}$.
Consider $N:=\otimes_{0\leq l<j}C_{2l+1}$ and $N':=\otimes_{j<l<d}C_{2l+1}$.
Then $N'$ is an extension of $T$ and $T(r)[*]$, where $r>j$. Since
$M=N\otimes C_{2j+1}\otimes N'$, we get an exact triangle:
$R\row M\row N\otimes C_{2j+1}\row R[1]$, where $R$ is an extension of
$T(r)[*]$ with $r>j$. In particular, $\Hom(M,T(j)[2j])=\Hom(N\otimes C_{2j+1},T(j)[2j])$, and
we have an exact sequence:
$$
\Hom(M,T(j)[2j])\row \Hom(N(j)[2j],T(j)[2j])\row \Hom(N,T(j)[2j+1]).
$$
Identifying $\Hom(N(j)[2j],T(j)[2j])$ with $\Hom(T(j)[2j],T(j)[2j])=\zz/2$ we get an exact sequence:
\begin{equation}
\label{ffipsi}
Hom(M,T(j)[2j])\stackrel{\ffi}{\row}\zz/2\stackrel{\psi}{\row} Hom(N,T(j)[2j+1]),
\end{equation}
where $\psi$ sends $1\in\zz/2$ to the composition $N\row T\stackrel{u_{2j+1}(q)}{\lrow}T(j)[2j+1]$,
and
$$
\Hom(M,T(j)[2j])=\op{CH}^j(G_d(q))/2
$$ (since the restriction of $M$ to $\dmDVA{\Field}$
is $M(G_d(q))$, and the restriction functor is a full embedding by \cite[Lemma 6.7]{VoMSS}).

From (\ref{hiXBOnDiagr}) we have the following commutative diagram:
$$
\xymatrix @-0.7pc{
X_q \ar @{->}[d] & \Spec\Field(X_q) \ar @{->}[l] \ar @{->}[r] & \bullet \ar @{->}[d] \\
X_q\times\ezar{O(n)} \ar @{->}[rr]  \ar @{->}[d] & & \ezar{O(n)} \ar @{->}[d]\\
\hiq{X_q} \ar @{->}[rr] & & \bzar{O(n)}.
}
$$

Let $\ov{C}_{2l+1}$ be the restriction of $C_{2l+1}$ to the category $\dmDVA{\Spec\Field(X_q)}$.
It follows from our diagram and Lemma \ref{z-l} that the natural lifting
$M_{\Field(X_q)}\row T(l)[2l]$ of the projection to $\ov{C}_{2l+1}$ is given by
$z_l\in\op{CH}^l(G_d(q)|_{\Field(X_q)})/2$. In particular, the map $\ffi: Hom(M,T(j)[2j])\row\zz/2$
will be surjective if and only if in $\op{CH}^j(G_d(q))/2$ there is an element whose restriction to
$\op{CH}^*(G_d(q)|_{\Field(X_q)})/2=\Lambda_{\zz/2}(z_0,\ldots,z_{d-1})$ (additive isomorphism -
see \cite[Proposition 2.4]{CGQG}) has
a non-zero $z_j$-coordinate. By \cite[Main Theorem 5.8]{CGQG} this is equivalent to: $z_j$ is defined
over $\Field$, or in other words, $j\in J(q)$.
It follows from (\ref{ffipsi}) that this condition is equivalent to the fact that the composition
$N\row\hiq{X_q}\stackrel{u_{2j+1}(q)}{\lrow}\hiq{X_q}(j)[2j+1]$ is zero.
\end{proof}

We immediately obtain:

\begin{cor}
\label{u->J}
Let $q$ be an $n$-dimensional quadratic form, and $p=q$, for even $n$, and $p=q\perp\la\ddet_{\pm}(q)\ra$,
for odd $n$. Then
$$
u_{2j+1}(p)\in (u_{2l+1}(p)\,|\,0\leq l<j)\cdot\hm^{*,*'}(\hiq{X_p},\zz/2)\hspace{3mm}\Rightarrow
\hspace{3mm} j\in J(q).
$$
\end{cor}

\begin{proof}
Since, for odd $n$, $J(q)=J(p)\backslash\{0\}$ - see \cite[Definition 5.11]{CGQG}, we can assume that $n=2d$ is even, and $p=q$.
Our result follows from Proposition \ref{zf} taking into account that the projection
$N_{j-1}\row\hiq{X_q}$ factors through
$$
C_{2l+1}=\op{Cone}[-1]\left(\hiq{X_q}\stackrel{u_{2l+1}(q)}{\lrow}\hiq{X_q}(l)[2l+1]\right),
$$
for each $0\leq l<j$.
\end{proof}


\begin{que}
\label{Jcond}
Are the following conditions equivalent?
\begin{itemize}
\item[$(1)$ ] $j\in J(q)$;
\item[$(2)$ ] $u_{2j+1}(p)=f(u_1(p),\ldots,u_{2j}(p))$, for some $f\in H[u_1,\ldots,u_{2j}]$;
\item[$(3)$ ] $u_{2j+1}(p)\in (u_{2i+1}(p)\,|\,0\leq i<j)\cdot\hm^{*,*'}(\hiq{X_p},\zz/2)$.
\end{itemize}
\end{que}


\begin{rem}
\label{nontrivJu}
Note, that the condition $j\in J(q)$ is not equivalent to $u_{2j+1}(p)=0$, even when $q\in I^2$.
Consider $q=\lva a,b\rva\cdot\la 1,c,d\ra$, where $a,b,c,d$ are "generic".
Then it follows from the computations of the Example \ref{<<a>>} that
$u_{11}(q)=\mu_{\{a,b\}}^3\cdot\{c,d\}\neq 0\in\hm^{*,*'}(\hiq{X_q},\zz/2)$
(by Proposition \ref{Hmalpha}, under the identification of the $6$-th diagonal in
$\hm^{*,*'}(\hiq{\{a,b\}},\zz/2)$ with the ideal in $K^M_*(k)/2$, this element corresponds
to $\{a,b,c,d\}\neq 0$). At the same time, $J(q)=\{0,\ldots,5\}\backslash 1$ contains $5$.
But $u_{11}(q)=u_3(q)\cdot u_8(q)$.
\end{rem}


We can compute the map $\ax{q}^*$ completely in the case of a Pfister form due
to the fact that it is the rare case where the motivic cohomology of $\hiq{X_q}$ is known.
The following computations were performed in the original version of \cite{OVV}, and
later by N.Yagita in \cite[Theorem 5.8]{Yap}.

\begin{thm}
\label{Hmalpha}
Let $\alpha=\{a_1,\ldots,a_n\}\in K^M_*(k)/2$, and $\hiq{\alpha}=\hiq{Q_{\alpha}}$, where
$Q_{\alpha}$ is a Pfister quadric corresponding to $\alpha$. Then the $\leq 0$ diagonal part
of $\hm^{*,*'}(\hiq{\alpha},\zz/2)$ is identified with $\hm^{*,*'}(\op{Spec}(k),\zz/2)$ by
the restriction $\hiq{\alpha}\row\op{Spec}(k)$. The $>0$ diagonal part of
$\hm^{*,*'}(\hiq{\alpha},\zz/2)$ as a $K^M_*(k)/2$-module
is isomorphic to
$$
\zz/2[\mu]\otimes\Lambda(Q_0,\ldots,Q_{n-2})(\gamma)\otimes
L,
$$
where $\Lambda$ is the external algebra (over $\zz/2$), $Q_i$
is the $i$-th Milnor operation (of degree $(2^i-1)[2^{i+1}-1]$),
$\gamma\in\hm^{n,n-1}(\hiq{\alpha},\zz/2)$ is the unique element such that
$\tau\cdot\gamma=\alpha$,
$\mu=Q_{n-2}\circ\ldots\circ Q_0(\gamma)\in\hm^{2^n-1,2^{n-1}-1}(\hiq{\alpha},\zz/2)$, and
multiplication by $\tau$ identifies the $1$-st diagonal $\gamma\otimes L$ with
$\alpha\cdot K^M_*(k)/2=\kker(K^M_*(k)/2\row K^M_*(k(Q_{\alpha}))/2)$.
Moreover, for any $Q_I\in \Lambda(Q_0,\ldots,Q_{n-2})$, $Q_{n-1}(Q_I(\gamma))=Q_I(\gamma)\cdot\mu$.
\end{thm}

Now it is not difficult to compute the subtle Stiefel--Whitney classes.

\begin{thm}
\label{Pfister}
Let $\alpha=\{a_1,\ldots,a_n\}\in K^M_n(k)/2$ be a non-zero pure symbol, and
$q_{\alpha}=\lva a_1,\ldots,a_n\rva$ be the respective Pfister form. Then
\begin{equation*}
u_i(q_{\alpha})=
\begin{cases}
&Q_{n-2}\circ\ldots\circ\widehat{Q_{r-1}}\circ\ldots\circ Q_0(\gamma_{\alpha}),\,\,\,
\text{if}\,\,\,i=2^n-2^r,\,0\leq r\leq n-1;\\
&0,\,\,\,\text{otherwise}
\end{cases}
\end{equation*}
\end{thm}

\begin{proof}
Induction on $n$.
(base) For $n=1$, we know that $u_1(q_{\{a\}})=\gamma_{\{a\}}$ and $u_2(q_{\{a\}})=0$.

(step) Consider $\beta=\{a_1,\ldots,a_{n-1}\}$, so that $\alpha=\beta\cdot\{a_n\}$.
Then we have the canonical (unique) map
$\hiq{\beta}\stackrel{f}{\row}\hiq{\alpha}$, and it follows from
Theorem \ref{Hmalpha} that $f^*(\gamma_{\alpha})=\gamma_{\beta}\cdot\{a_n\}$.
We have: $q_{\alpha}=q_{\beta}\perp -a_n\cdot q_{\beta}$. We have the respective
embedding $O(2^{n-1})\times O(2^{n-1})\stackrel{j}{\hookrightarrow} O(2^n)$, and
by Proposition \ref{Olmn} and $(\ref{YXalpha})$, we get:
$$
f^*(u_i(q_{\alpha}))=
\sum_{j=0}^iu_j(q_{\beta})\cdot u_{i-j}(-a_n\cdot q_{\beta})\cdot\tau^{[i/2]-[j/2]-[i-j/2]}.
$$
By Proposition \ref{lambdadejstvie} and inductive assumption,
$u_l(-a_n\cdot q_{\beta})=u_l(q_{\beta})$, for $l<2^{n-1}$, while
$u_{2^{n-1}}(-a_n\cdot q_{\beta})=\{a_n\}\cdot u_{2^{n-1}-1}(q_{\beta})$.
This implies that $f^*(u_i(q_{\alpha}))=0$, if $i\neq 2^{n}-2^r$, for $0\leq r\leq n-1$, and
\begin{equation*}
\begin{split}
& f^*(u_{2^{n}-2^r}(q_{\alpha}))= u_{2^{n-1}-2^r}(q_{\beta})\cdot u_{2^{n-1}-1}(q_{\beta})\cdot\{a_n\}=\\
&= Q_{n-3}\circ\ldots\circ\widehat{Q_{r-1}}\circ\ldots\circ Q_0(\gamma_{\beta})\cdot\mu_{\beta}\cdot\{a_n\}
=f^*(Q_{n-2}\circ\ldots\circ\widehat{Q_{r-1}}\circ\ldots\circ Q_0(\gamma_{\alpha})),
\end{split}
\end{equation*}
for $0\leq r\leq n-2$.

For $r=n-1$, $f^*(u_{2^{n-1}}(q_{\alpha}))=u_{2^{n-1}-1}(q_{\beta})\cdot\{a_n\}=
Q_{n-3}\circ\ldots\circ Q_0(\gamma_{\beta})\cdot\{a_n\}=f^*(Q_{n-3}\circ\ldots\circ Q_0(\gamma_{\alpha}))$.
Since $f^*$ is injective on all the diagonals up to $2^{n-1}$
(follows from Theorem \ref{Hmalpha}), we obtain:
$u_i(q_{\alpha})=0$, for $i\neq 2^{n}-2^r$, for $0\leq r\leq n-1$ and
$u_{2^n-2^r}(q_{\alpha})= Q_{n-2}\circ\ldots\circ\widehat{Q_{r-1}}\circ\ldots\circ Q_0(\gamma_{\alpha})$
(recall that $u_i$ lives on the diagonal with number $[i+1/2]$).
\end{proof}

Now we can use {\it subtle Stiefel--Whitney classes} to describe the powers of the fundamental ideal
$I^n$ in $W(k)$.

\begin{thm}
\label{In}
The following conditions are equivalent:
\begin{itemize}
\item[$1)$ ] $q\in I^n$;
\item[$2)$ ] $u_i(q)=0$, for $1\leq i\leq 2^{n-1}-1$;
\item[$3)$ ] $u_i(q)=0$, for $i=2^r,\,0\leq r\leq n-2$.
\end{itemize}
\end{thm}

\begin{proof} The implication $(2\row 3)$ is evident.
$(3\row 1)$: Suppose, $q\not\in I^n$. Then by \cite[Theorem 4.3]{OVV} (the $J$-filtration Conjecture),
there exists a field extension $L/k$ such that $(q_L)_{an}$ is an $r$-fold Pfister form, with
$r<n$. By Theorem \ref{Pfister}, $u_{2^{r-1}}(q)|_L\neq 0$ - a contradiction.

$(1\row 2)$: We will show that the respective cohomology group is zero.


\begin{prop}
\label{cohIn}
Let $q_1,\ldots,q_s$ be forms from $I^n$. Then
$$
\hm^{b,a}(\hiq{X_{q_1}}\times\ldots\times\hiq{X_{q_s}},\zz/2)=0\hspace{5mm}\text{ for }:
$$
\begin{itemize}
\item[$1)$ ] $\frac{b}{a}>2+\frac{1}{2^{n-1}-1}$; \,\,\,\text{and for}
\item[$2)$ ] $\frac{b+l-n+1}{a+l-n+1}>2+\frac{1}{2^{l}-1}$, where
$n-1\geq l=[log_2(b-a)]$ and $b>a$.
\end{itemize}
\end{prop}

\begin{proof} Let $\car$ denote the set of pairs $(a,b)$ satisfying the
conditions 1), or 2) of Proposition \ref{cohIn} union with
$\{(a,b)|\,b\leq a\}$. Denote as $\dcar$ the set of such $(a,b)$ that $(a,b)\notin\car$, but
$(a,b+1)\in\car$.

We can safely assume that $n>1$. Use increasing induction on $a$. For $a<0$, the groups in question
are, clearly, zero.


\begin{lem}
\label{hiqtildeIn}
Let $p_1,\ldots,p_s\in I^n$, and $q_{\alpha}$ be an $n$-fold Pfister form.
Suppose $(a,b)\in\car$, and Proposition \ref{cohIn} is valid for all $(a',b')\in\car$ with $b'>a'<a$.
Then the natural map
$$
\hm^{b,a}(\hiq{X_{p_1}}\times\ldots\times\hiq{X_{p_s}},\zz/2)\row
\hm^{b,a}(\hiq{X_{p_1}}\times\ldots\times\hiq{X_{p_s}}\times\hiq{\alpha},\zz/2)
$$
is an isomorphism for the given $(a,b)$.
\end{lem}

\begin{proof}
In $\dmkDVA$ we have a distinguished triangle:
$\hiqt{\alpha}\row\hiq{\alpha}\row\zz/2\row\hiqt{\alpha}[1]$.

\begin{lem}
\label{XhiqtildeQ}
Let $Y$ be any smooth variety, and $q_{\alpha}$ be an $n$-fold Pfister form.
Then $\hm^{b,a}(Y\times\hiqt{\alpha},\zz/2)=0$, for $(a,b)\in\car$, and
the map $\hm^{b,a}(Y\times\hiq{\alpha},\zz/2)\twoheadrightarrow\hm^{b,a}(Y\times\hiqt{\alpha},\zz/2)$
is surjective for $(a,b)\in\dcar$. In particular, the map
$$
\hm^{b,a}(Y,\zz/2)\row\hm^{b,a}(Y\times\hiq{\alpha},\zz/2)
$$
is an isomorphism, for $(a,b)\in\car$.
\end{lem}

\begin{proof}
For any field extension $L/k$, $\hm^{b,a}(\hiqt{\alpha}|_L,\zz/2)=0$ for $(a,b)\in\car$,
by Theorem \ref{Hmalpha} (the $b\leq a$ case follows already from the
Beilinson-Lichtenbaum Conjecture proven by V.Voevodsky \cite{VoMil} using
A.Suslin-V.Voevodsky \cite{SuVo}).
Since $\car$ is stable under $(a,b)\mapsto(a-m,b-2m)$, for any $m\geq 0$,
it follows from the localization sequence for motivic cohomology that
$\hm^{b,a}(Y\times\hiqt{\alpha},\zz/2)=0$, for any smooth variety $Y$, for $(a,b)\in\car$.
Theorem \ref{Hmalpha} implies also that, for any field extension $L/k$,
the map $\hm^{b,a}(\hiqt{\alpha},\zz/2)\twoheadrightarrow\hm^{b,a}(\hiqt{\alpha}|_L,\zz/2)$
is surjective, for $(a,b)\in\dcar$. Since the map
$\hm^{b,a}(\hiq{\alpha},\zz/2)\twoheadrightarrow\hm^{b,a}(\hiqt{\alpha},\zz/2)$ is surjective,
for all $(a,b)$, it follows from the localization sequence
(and the fact that $\car$ is stable under: $(a,b)\mapsto(a-2m,b-2m)$) again that the map
$\hm^{b,a}(Y\times\hiq{\alpha},\zz/2)\twoheadrightarrow\hm^{b,a}(Y\times\hiqt{\alpha},\zz/2)$
is surjective for $(a,b)\in\dcar$.
\end{proof}

Let $q$ be one of our forms $p_1,\ldots,p_s$. We know that $\hiq{X_q}=\hiq{G_d(q)}$, where $G_d(q)$ is
the variety of totally isotropic subspaces in $V_q$ of maximal dimension.
By Theorem \ref{GdhiqEV}, $M(G_d(q))$ is an extension of $M(\hiq{G_d(q)})(j)[2j]$,
for some $j$'s.
We have the following description of $I^n$ in terms of the $J$-invariant.

\begin{prop}
\label{JIn}
The following conditions are equivalent:
\begin{itemize}
\item[$1)$ ] $q\in I^n$;
\item[$2)$ ] $\{0,\ldots,2^{n-1}-2\}\subset J(q)$;
\item[$3)$ ] $2^r-1\in J(q)$, for $0\leq r\leq n-2$.
\end{itemize}
\end{prop}

\begin{proof}
$(2\row 3)$ is evident. $(1\row 2)$ follows from \cite[Corollary 3.5]{GPQCG}.
$(3\row 1)$ follows from the J-filtration Conjecture (\cite[Theorem 4.3]{OVV}) and
\cite[Example 3.2]{u}.
\end{proof}

It follows from Proposition \ref{minJU} and Proposition \ref{JIn} that $u_{2j+1}(q)=0$, for
$0\leq j<2^{n-1}-1$. This implies that we have a direct summand $N$ of
$M(G_d(q))$ such that
$N_{\kbar}=\otimes_{j}(\zz/2\oplus\zz/2(j)[2j])$, where $j\geq 2^{n-1}-1$.

Then we have a distinguished triangle:
$P\row N\row M(\hiq{F_q})\row P[1]$ in $\dmkDVA$, where
$P$ is an extension of
$M(\hiq{F_q})(j)[2j]$, for $j\geq 2^{n-1}-1$.
Thus, if we know that\\
$\hm^{b,a}(F_{p_1}\times\ldots\times F_{p_s}\times\hiqt{\alpha},\zz/2)=0$,
and $\hm^{b-2j-1,a-j}(\hiq{F_{p_1}}\times\ldots\times\hiq{F_{p_s}}\times\hiqt{\alpha},\zz/2)=0$,
for all $j\geq 2^{n-1}-1$, then
$\hm^{b,a}(\hiq{F_{p_1}}\times\ldots\times\hiq{F_{p_s}}\times\hiqt{\alpha},\zz/2)=0$.
Hence, it follows by the induction on the degree from Lemma \ref{XhiqtildeQ}
and the fact that $\car$ is stable under $(a,b)\mapsto (a-j,b-2j-1)$, for
$j\geq 2^{n-1}-1$, that
$\hm^{b,a}(\hiq{F_{p_1}}\times\ldots\times\hiq{F_{p_s}}\times\hiqt{\alpha},\zz/2)=0$,
for $(a,b)\in\car$.
By the inductive assumption, $\hm^{b,a}(P,\zz/2)=0$, for our pair $(a,b)\in\car$.
Thus, the map
$\hm^{b,a}(\hiq{F_{p_1}}\times\ldots\times\hiq{F_{p_s}},\zz/2)\hookrightarrow
\hm^{b,a}(F_{p_1}\times\ldots\times F_{p_s},\zz/2)$ is injective, and so,
by Lemma \ref{XhiqtildeQ}, the map
$\hm^{b,a}(\hiq{F_{p_1}}\times\ldots\times\hiq{F_{p_s}},\zz/2)\row
\hm^{b,a}(\hiq{F_{p_1}}\times\ldots\times\hiq{F_{p_s}}\times\hiq{\alpha},\zz/2)$
is an isomorphism. This proves Lemma \ref{hiqtildeIn}.
\end{proof}

We can present each of the forms $p_1,\ldots,p_s$ as sum of $n$-fold Pfister forms.
Let $\pi_{\alpha_1},\ldots,\pi_{\alpha_t}$ be all the Pfister forms involved.
Then it follows from Lemma \ref{hiqtildeIn} that (for the given pair $(a,b)$) the map:
$$
\hm^{b,a}(\hiq{X_{p_1}}\times\ldots\times\hiq{X_{p_s}},\zz/2)\row
\hm^{b,a}(\hiq{X_{p_1}}\times\ldots\times\hiq{X_{p_s}}\times
\hiq{\alpha_1}\times\ldots\times\hiq{\alpha_t},\zz/2)
$$
is an isomorphism. But
$\hiq{X_{p_1}}\times\ldots\times\hiq{X_{p_s}}\times\hiq{\alpha_1}\times\ldots\times\hiq{\alpha_t}=
\hiq{\alpha_1}\times\ldots\times\hiq{\alpha_t}$.
It remains to apply Lemma \ref{hiqtildeIn} again reducing the set $\alpha_1,\ldots,\alpha_t$
to an empty one.
Proposition \ref{cohIn} is proven.
\end{proof}

It follows from Proposition \ref{cohIn} that $u_i(q)$, for $i\leq 2^{n-1}-1$, lives in a zero
group, which proves the implication $(1\row 2)$.
\end{proof}

The above results imply that {\it subtle Stiefel--Whitney classes} do distinguish the
triviality of the torsor.

\begin{cor}
\label{Utriv}
The following conditions are equivalent:
\begin{itemize}
\item[$1)$ ] $q\cong q_n$;
\item[$2)$ ] $u_i(q)=0$, for all $i$;
\item[$3)$ ] $u_{2^r}(q)=0$, for all $r$;
\item[$4)$ ] $u_{2^r-1}(q)=0$, for all $r$.
\end{itemize}
\end{cor}

\begin{proof}
The implications $(1\row 2)$, $(2\row 3)$, and $(2\row 4)$ are evident.
On the other hand, by adding $\la 1\ra$ to our form, if needed, we can assume that it is even-dimensional.
If $q\not\cong q_n$, then using the tower of M.Knebusch (see \cite{Kn1}) we can find a field
extension $L/k$ such that $(q_L)_{an}$ is a Pfister form (using the result of A.Pfister:
any even-dimensional form of height one is proportional to a Pfister form - see \cite{Pf}).
Then, by Theorem \ref{Pfister}, $u_{2^r}(q_L)\neq 0$, and $u_{2^{r+1}-1}(q_L)\neq 0$, for some $r$.
This proves $(3\row 1)$ and $(4\row 1)$.
\end{proof}

Although, the {\it subtle Stiefel--Whitney classes} distinguish the triviality of the torsor, these do not,
unfortunately, distinguish torsors among themselves.

\begin{exa}
\label{<<a>>}
Let $q_{\alpha}=\lva a_1,\ldots,a_d\rva$ be a $d$-fold anisotropic Pfister form, and
$p$ be an odd-dimensional form. Consider $q=q_{\alpha}\cdot p$. Then $q$ is even-dimensional, and,
for any field extension $L/k$,
$$
(X_q)|_L\,\,\,\text{is trivial}\LRw \alpha|_L=0\in K^M_d(k)/2.
$$
Thus, $\hiq{X_q}=\hiq{\alpha}$.
Moreover, if $p=p_1\perp p_2$, and $q_i=q_{\alpha}\cdot p_i$, then
$\hiq{X_{q_1}}\times\hiq{X_{q_2}}=\hiq{\alpha}$.
If $p=\la b_1,\ldots,b_m\ra$, then, by
the proof of Theorem \ref{Pfister},
$\sum_iu_i(b_l\cdot q_{\alpha})=\sum_{s=0}^{n-2}Q_{[0,\ldots,n-2]\backslash s}(\gamma_{\alpha})+
\mu_{\alpha}\cdot (1+\{b_l\})$.
It is also known (is contained in the original version of \cite{OVV}) that
(provided $-1$ is a square in $k$):
\begin{equation*}
Q_I(\gamma_{\alpha})\cdot Q_J(\gamma_{\alpha})=
\begin{cases}
&\mu_{\alpha}\cdot Q_{I\cap J}(\gamma_{\alpha}),\,\,\,\text{if}\,\,\,I\cup J=\{0,\ldots,n-2\};\\
&0\,\,\,\text{otherwise}.
\end{cases}
\end{equation*}
By Theorem \ref{Hmalpha} each positive diagonal of $\hm^{*,*'}(\hiq{\alpha},\zz/2)$ can
be identified with $\alpha\cdot K^M_*(k)/2\subset K^M_*(k)/2$. It follows from Proposition \ref{Olmn},
that under this identification, {\it subtle Stiefel--Whitney classes} of $q$ are identified with
(some multiples of)
$\alpha\cdot\omega_j(p)$, for some $j$.
Notice, that it is more "informative" than $\omega_l(q)$!
At the same time, if $p-p'\in I^3$, then {\it subtle Stiefel--Whitney classes} of $q=q_{\alpha}\cdot p$
and $q'=q_{\alpha}\cdot p'$ will be the same.
\end{exa}

But subtle Stiefel--Whitney classes carry a lot of information about $q$. In particular, they
contain {\it Arason invariant} (see \cite{Ar}) and all higher invariants $e_r:I^r/I^{r+1}\row K^M_r(k)/2$.
Indeed, since, for any $q\in I^n$,  $u_{2j+1}(q)=0$, for all $j<2^{n-1}-1$, it follows from
Theorem \ref{GdhiqEV} and the fact that $\hm^{b,a}(G_d(q),\zz/2)=0$, for all $b>2a$, that
$\hm^{2^n-1,2^{n-1}-1}(\hiq{X_q},\zz/2)=\zz/2\cdot u_{2^n-1}(q)$.
On the other hand, by \cite[Proposition 2.3]{OVV}, we have an exact sequence:
$$
0\row\hm^{n,n-1}(\hiq{X_q},\zz/2)\stackrel{\cdot\tau}{\lrow}K^M_n(k)/2\row K^M_n(k(X_q))/2.
$$
Clearly, $e_n(q)\in\kker(K^M_n(k)/2\row K^M_n(k(X_q))/2)=\kker(K^M_n(k)/2\row K^M_n(k(F_d(q)))/2)$,
and it follows from \cite[Theorems 3.2 and 4.2]{OVV} that
$\hm^{n,n-1}(\hiq{X_q},\zz/2)=\zz/2\cdot\gamma$, where $\tau\cdot\gamma=e_n(q)$.
By the $J$-filtration Conjecture (\cite[Theorem 4.3]{OVV}), there exists such field extension $L/k$
that $(q_L)_{an}$ is an $n$-fold Pfister form $q_{\alpha}$ (if $q\notin I^{n+1}$).
And we know from Theorem \ref{Pfister} that
$u_{2^n-1}(q_{\alpha})=Q_{n-2}\circ\ldots\circ Q_0(\gamma_{\alpha})$.
Hence, the same is true for $q$. Thus, we obtain:


\begin{thm}
\label{en}
Let $q\in I^n$. Then the map
$$
Q_{n-2}\circ\ldots\circ Q_0:\hm^{n,n-1}(\hiq{X_q},\zz/2)\stackrel{\cong}{\lrow}
\hm^{2^n-1,2^{n-1}-1}(\hiq{X_q},\zz/2)
$$
is an isomorphism, and
$$
(Q_{n-2}\circ\ldots\circ Q_0)^{-1}(u_{2^n-1}(q))\cdot\tau=e_n(q).
$$
\end{thm}

\begin{rem}
\label{en2^n-1}
One can also show that $e_n(q)=(Q_{n-3}\circ\ldots\circ Q_0)^{-1}(u_{2^{n-1}}(q))\cdot\tau$,
but it requires a bit more of work.
\end{rem}



Alexander Smirnov

smirnov@pdmi.ras.ru

St. Petersburg Departement of Steklov Math. Institute
\\

Alexander Vishik

alexander.vishik@nottingham.ac.uk

School of Mathematical Sciences, University of Nottingham

\end{document}